\documentclass
[
    a4paper,
    DIV=11,
    abstracton,
]
{scrartcl}

\usepackage
{
    graphicx,
    amssymb,
    amsmath,
    amsthm,
    dsfont, 
    xcolor,
    mathtools,
    authblk,
    enumitem,
    tikz,
    mdframed, 
}

\usepackage[utf8]{inputenc}

\usepackage[pdffitwindow=false,
            plainpages=false,
            pdfpagelabels=true,
            pdfpagemode=UseOutlines,
            pdfpagelayout=SinglePage,
            bookmarks=false,
            colorlinks=true,
            hyperfootnotes=false,
            linkcolor=blue,
            urlcolor=blue!30!black,
            citecolor=green!50!black]{hyperref}

\usepackage[bf,normal]{caption}

\graphicspath{{pics/}}

\newcommand{\subfiguretitle}[1]{{\scriptsize{#1}} \\}
\newcommand{\R}{\mathbb{R}}                                     
\newcommand{\pd}[2]{\frac{\partial#1}{\partial#2}}              
\newcommand{\innerprod}[2]{\left\langle #1,\, #2 \right\rangle} 
\newcommand{\ts}{\hspace*{0.1em}}                               
\providecommand{\abs}[1]{\left\lvert #1 \right\rvert}           
\providecommand{\norm}[1]{\left\lVert #1 \right\rVert}          

\newcommand\xqed[1]{\leavevmode\unskip\penalty9999 \hbox{}\nobreak\hfill \quad\hbox{#1}}
\newcommand{\exampleSymbol}{\xqed{$\triangle$}}

\DeclareMathOperator{\mspan}{span}

\newtheorem{theorem}{Theorem}[section]
\newtheorem{corollary}[theorem]{Corollary}
\newtheorem{lemma}[theorem]{Lemma}
\newtheorem{proposition}[theorem]{Proposition}
\newtheorem{definition}[theorem]{Definition}
\theoremstyle{definition}
\newtheorem{example}[theorem]{Example}
\newtheorem{remark}[theorem]{Remark}
\newtheorem{textalgorithm}[theorem]{Algorithm}

\makeatletter
\renewcommand*\env@matrix[1][*\c@MaxMatrixCols c]{%
  \hskip -\arraycolsep
  \let\@ifnextchar\new@ifnextchar
  \array{#1}}
\makeatother

\mathtoolsset{centercolon} 

\allowdisplaybreaks

\makeatletter
\def\blfootnote{\gdef\@thefnmark{}\@footnotetext}
\makeatother

\definecolor{boxback}{gray}{0.95}

\begin{document}

\title{Kernel-based approximation of the \\ Koopman generator and Schrödinger operator}
\author[1]{Stefan Klus}
\author[2]{Feliks Nüske}
\author[3]{Boumediene Hamzi}
\affil[1]{Department of Mathematics and Computer Science, Freie Universität Berlin, Germany}
\affil[2]{Department of Mathematics, Paderborn University, Germany}
\affil[3]{Department of Mathematics, Imperial College London, United Kingdom}

\date{}

\maketitle


\begin{abstract}
Many dimensionality and model reduction techniques rely on estimating dominant eigenfunctions of associated dynamical operators from data. Important examples include the Koopman operator and its generator, but also the Schrödinger operator. We propose a kernel-based method for the approximation of differential operators in reproducing kernel Hilbert spaces and show how eigenfunctions can be estimated by solving auxiliary matrix eigenvalue problems. The resulting algorithms are applied to molecular dynamics and quantum chemistry examples. Furthermore, we exploit that, under certain conditions, the Schrödinger operator can be transformed into a Kolmogorov backward operator corresponding to a drift-diffusion process and vice versa. This allows us to apply methods developed for the analysis of high-dimensional stochastic differential equations to quantum mechanical systems.
\end{abstract}

\section{Introduction}

The Koopman operator \cite{Ko31, LaMa94, Mezic05, BMM12} plays a central role in the global analysis of complex dynamical systems. It is, for instance, used to find conformations of molecules, coherent patterns in fluid flows, but also for prediction, stability analysis, and control \cite{MaMe16, KKS16, Kaiser17, KM18a, PK19, KHMN19}. Instead of analyzing a given finite-dimensional but highly nonlinear system directly, the underlying idea is compute an associated infinite-dimensional but linear operator \cite{BMM12}. By computing an approximation of this operator from measurement or simulation data, it is possible to extract Koopman eigenvalues, eigenfunctions, and modes. The most frequently used techniques are based on variants or generalizations of \emph{extended dynamic mode decomposition} (EDMD) \cite{WKR15, WRK15}. A reformulation of EDMD for the generator of the Koopman operator, called gEDMD, was recently proposed in \cite{KNPNCS20}. It was shown that in addition to the previously mentioned applications, the generator contains valuable information about the governing equations of a system, see also \cite{MauGon16, Kaiser17}. System identification aims at learning a preferably parsimonious model from data. That is, the learned model should comprise as few terms as possible and still have predictive power, which is typically accomplished by utilizing sparse regression techniques. One drawback of gEDMD is that it requires a set of explicitly chosen basis functions and their first- and---if the system is non-deterministic and non-reversible---second-order derivatives. Moreover, the size of the resulting matrix eigenvalue problem that needs to be solved to compute eigenvalues, eigenfunctions, and modes of the generator depends on the size of the dictionary. The goal of this paper is to derive a kernel-based method to approximate the Koopman generator from data. A kernel-based variant of EDMD has been proposed in \cite{WRK15} and generalized in \cite{KSM19}. We derive a kernel-based variant of gEDMD. Employing the well-known kernel trick, a dual eigenvalue problem whose size depends on the number of snapshots can be constructed. The resulting methods allow for implicitly infinite-dimensional feature spaces and only require partial derivatives of the kernel function. This enables us to apply the methods to high-dimensional systems for which conventional techniques would be prohibitively expensive due to the curse of dimensionality, provided the number of snapshots is such that the eigenvalue problem can still be solved numerically or can be downsampled without losing essential information. Since we aim at approximating differential operators, we need to be able to represent derivatives in reproducing kernel Hilbert spaces. This requires the notion of derivative reproducing properties. Derivative reproducing kernels\footnote{Reproducing kernel Hilbert spaces with derivative reproducing properties are related to \emph{Native Spaces} introduced in a different context in \cite{wendland_2004}.} \cite{Zhou08} have been used to approximate Lyapunov functions for ordinary differential equations in \cite{lyap_bh} and to approximate center manifolds for ordinary differential equations in \cite{kernels_center_manifolds}.

Similar operators are also used for manifold learning and understanding the geometry of high-dimensional data \cite{Coifman06, NLCK06, CKLMN08, NLCK08}. Methods like \emph{diffusion maps} construct graph Laplacians with the aid of diffusion kernels, effectively approximating transition probabilities between data points. In the infinite-data limit and letting the kernel bandwidth go to zero, it has been shown that these methods, depending on the normalization, essentially compute eigenfunctions of certain differential operators, e.g., the Laplace--Beltrami operator, the Kolmogorov backward operator, or the Fokker--Planck operator.

Another related differential operator that is of utmost importance in quantum mechanics is the Schrödinger operator. Solutions of the time-independent Schrödinger equation describe stationary states and associated energy levels. We will illustrate how kernel-based methods developed for the Koopman generator can be applied to these related problems. The main contributions of this paper are:
\begin{itemize}[wide, itemindent=\parindent, itemsep=0ex, topsep=0.5ex]
\item We show how the derivative reproducing properties of kernels can be used to approximate differential operators such as the Koopman generator and the Schödinger operator as well as their eigenvalues and eigenfunctions from data. Additionally, we derive a kernel-based method tailored to reversible dynamics, which does not require estimating drift and diffusion terms but only an equilibrated trajectory.
\item  Furthermore, we exploit the fact that, under certain conditions, the Schrödinger operator can be turned into a Kolmogorov backward operator (see, e.g., \cite{Pav14}), which allows for the interpretation of a quantum-mechanical system as a drift-diffusion process and, as a consequence, the application of methods developed for the analysis of stochastic differential equations or their generators.
\item We demonstrate potential applications in molecular dynamics, using the example of a quadruple-well problem, and quantum mechanics, describing how to apply the proposed methods directly to the Schrödinger equation or the associated stochastic process. This will be illustrated with two well-known examples, the quantum harmonic oscillator and the hydrogen atom.
\end{itemize}

The remainder of the manuscript is structured as follows: We first introduce the necessary tools, namely the Koopman operator, its generator, and (derivative) reproducing kernel Hilbert spaces in Section~\ref{sec:Koopman and RKHS}. Additionally, relationships with the Schrödinger equation will be explored. The derivation of the kernel-based formulation of gEDMD will be detailed in Section~\ref{sec:kgEDMD}. In Section~\ref{sec:Applications}, we will show how the derived methods can be applied to molecular dynamics and quantum mechanics problems. Concluding remarks and future work will be discussed in Section~\ref{sec:Conclusion}.

\section{Koopman theory and reproducing kernel Hilbert spaces}
\label{sec:Koopman and RKHS}

We start directly with the non-deterministic setting, the Koopman operator and its generator for ordinary differential equations can then be regarded as a special case, see also \cite{KNPNCS20} for a detailed comparison. The notation used below is summarized in Table~\ref{tab:Notation}.

\begin{table}[tb]
    \centering
    \caption{Overview of notation.}
    \begin{tabular}{ll}
        \hline
        $ X_t $                            & stochastic process \\
        $ \mathbb{X}  $                    & state space \\
        $ k, \phi $                        & kernel and associated feature map \\
        $ \mathbb{H} $                     & reproducing kernel Hilbert space induced by $ k $ \\
        $ \mathcal{K}^t $                  & Koopman operator with lag time $ t $ \\
        $ \mathcal{L} $                    & generator of the Koopman operator \\
        $ \mathcal{H} $                    & Schrödinger operator \\
        $ \mathcal{T} $                    & general differential operator \\
        $ \mathcal{T}_\mathbb{H} $         & kernel-based differential operator \\
        $ \mathcal{C}_{00} $               & covariance operator \\
        $ \widehat{\mathcal{A}} $          & empirical estimate of operator $ \mathcal{A} $ \\
        $ G_0, G_1, G_2 $                  & (generalizations of) Gram matrices \\
        \hline
    \end{tabular}
    \label{tab:Notation}
\end{table}

\subsection{The Koopman operator and its generator}

In what follows, let $ \mathbb{X} \subset \R^d $ be the state space and $ f \colon \mathbb{X} \to \R $ a real-valued observable of the system. Furthermore, let $ \mathbb{E}[\,\cdot\,] $ denote the expected value
and $ \Theta^t $ the flow map associated with a dynamical system, i.e., $ \Theta^t(X_0) = X_t $. Given a stochastic differential equation of the form
\begin{equation} \label{eq:SDE}
    \mathrm{d}X_t = b(X_t) \ts \mathrm{d}t + \sigma(X_t) \ts \mathrm{d}B_t,
\end{equation}
where $ b \colon \R^d \to \R^d $ is called the drift term, $ \sigma \colon \R^d \to \R^{d \times d} $ the diffusion term, and $ B_t $ is $ d $-dimensional Brownian motion, the stochastic Koopman operator is defined by
\begin{equation*} 
    (\mathcal{K}^t f)(x) = \mathbb{E}[f(\Theta^t(x))].
\end{equation*}
The infinitesimal generator $ \mathcal{L} $ of the semigroup of Koopman operators is given by
\begin{equation} \label{eq:generator_SDE}
    \mathcal{L} f = \sum_{i=1}^d b_i \ts \pd{f}{x_i} + \frac{1}{2} \sum_{i=1}^d \sum_{j=1}^d a_{ij} \ts \pd{^2 f}{x_i \ts \partial x_j}
\end{equation}
and its adjoint, the generator of the Perron--Frobenius operator, by
\begin{equation*}
    \mathcal{L}^* f = -\sum_{i=1}^d \pd{(b_i \ts f)}{x_i}  + \frac{1}{2} \sum_{i=1}^d \sum_{j=1}^d \pd{^2 (a_{ij} \ts f)}{x_i \ts \partial x_j},
\end{equation*}
with $ a = \sigma \ts \sigma^\top $. We assume from now on that $a$ is uniformly positive definite on $\mathbb{X}$. The second-order partial differential equation $ \pd{u}{t} = \mathcal{L} u $ is also called the \emph{Kolmogorov backward equation} and $ \pd{u}{t} = \mathcal{L}^* u $ the \emph{Fokker--Planck equation}~\cite{LaMa94}.

\begin{remark} \label{rem:overdamped_langevin}
As in \cite{KNPNCS20}, we will often consider systems of the form
\begin{equation*}
    \mathrm{d}X_t = -\nabla V(X_t) \ts \mathrm{d}t + \sqrt{2 \beta^{-1}} \ts \mathrm{d}B_t,
\end{equation*}
where $ V $ is a given potential and $ \beta $ the inverse temperature. In this case, the operators can be written as
\begin{equation*}
    \mathcal{L} f = -\nabla V \cdot \nabla f + \beta^{-1} \Delta f
    \quad \text{and} \quad
    \mathcal{L}^* f = \nabla V \cdot \nabla f + \Delta V \ts f + \beta^{-1} \Delta f.
\end{equation*}
\end{remark}

\subsection{Generator EDMD}

A data-driven method for the approximation of the generator of the Koopman operator and Perron--Frobenius operator called \emph{generator extended dynamic mode decomposition} or, in short, gEDMD was derived in \cite{KNPNCS20}. While standard EDMD requires a training data set $ \{\ts x_m \ts\}_{m=1}^M $ and the corresponding data points $\{\ts y_m \ts\}_{m=1}^M $, where $ y_m = \Theta^\tau(x_m) $ for a fixed lag time $ \tau $, gEDMD assumes that we can evaluate or estimate\footnote{Using, for instance, Kramers--Moyal formulae.} $ \{\ts b(x_m) \ts\}_{m=1}^M $ and $ \{\ts \sigma(x_m) \ts\}_{m=1}^M $. Choosing a dictionary of basis functions $ \{\ts \phi_n \ts\}_{n=1}^N $, where $ \phi_n \colon \R^d \to \R $, and defining $ \phi(x) = [\phi_1(x), \dots, \phi_N(x)]^\top $, we compute the matrices $ \Phi_X, \mathrm{d}\Phi_X \in \R^{N \times M} $, with
\begin{equation*}
    \Phi_X =
    \begin{bmatrix}
        \phi_1(x_1) & \dots  & \phi_1(x_M) \\
        \vdots      & \ddots & \vdots      \\
        \phi_N(x_1) & \dots  & \phi_N(x_M)
    \end{bmatrix}
    \quad \text{and} \quad
    \mathrm{d}\Phi_X =
    \begin{bmatrix}
        \mathrm{d}\phi_1(x_1) & \dots  & \mathrm{d}\phi_1(x_M) \\
        \vdots                & \ddots & \vdots                \\
        \mathrm{d}\phi_N(x_1) & \dots  & \mathrm{d}\phi_N(x_M)
    \end{bmatrix},
\end{equation*}
where
\begin{equation*}
    \mathrm{d}\phi_n(x) = \sum_{i=1}^d b_i(x) \ts \pd{\phi_n}{x_i}(x) + \frac{1}{2} \sum_{i=1}^d \sum_{j=1}^d a_{ij}(x) \ts \pd{^2 \phi_n}{x_i \ts \partial x_j}(x).
\end{equation*}
The matrix representation of the least-squares approximation of the Koopman generator $ \mathcal{L} $ is then given by
\begin{equation*}
    \widehat{L}^\top = \mathrm{d}\Phi_X \Phi_X^+ = \widehat{A} \ts \widehat{G}^+,
\end{equation*}
with
\begin{equation*}
    \widehat{A} = \frac{1}{M} \sum_{m=1}^M \mathrm{d}\phi(x_m) \ts \phi(x_m)^\top
    \quad \text{and} \quad
    \widehat{G} = \frac{1}{M} \sum_{m=1}^M \phi(x_m) \ts \phi(x_m)^\top.
\end{equation*}
It was shown that gEDMD, in the infinite-data limit, converges to a Galerkin projection of the generator onto the space spanned by the basis functions $ \{\ts \phi_n \ts\}_{n=1}^N $ and that $ \widehat{L} $ is an empirical estimate of the projected generator \cite{KNPNCS20}. Approximations of eigenfunctions of $ \mathcal{L} $ are then given by
\begin{equation*}
    \varphi_\ell(x) = \innerprod{\xi_\ell}{\phi(x)},
\end{equation*}
where $ \xi_\ell $ is an eigenvector of $ \widehat{L} $ corresponding to the eigenvalue $ \lambda_\ell $, and $\innerprod{\cdot}{\cdot} $ denotes the standard Euclidean inner product. Analogously, the generator of the Perron--Frobenius operator is given by $ (\widehat{L}^*)^\top = \widehat{A}\ts^\top \widehat{G}^{+} $. Further details, examples, and different applications including system identification, coarse graining, and control can also be found in \cite{KNPNCS20}.

\subsection{Second-order differential operators}
\label{sec:Dynamical operators}

Consider the generator $\mathcal{L}$ in \eqref{eq:generator_SDE}, and assume there is a unique strictly positive invariant density $\rho_0$, which we can write as $\rho_0(x) \propto \exp(-F(x))$. The function $F$ is called a generalized potential (with $F = \beta V$ for the SDE in Remark \ref{rem:overdamped_langevin}). The measure corresponding to $\rho_0$ is denoted by $ \mathrm{d} \mu = \rho_0 \ts \mathrm{d} x $. The negative generator can be decomposed into a symmetric and an anti-symmetric part as
\begin{align} \label{eq:generator_decomposition}
-\mathcal{L} &= -\frac{1}{2}e^{F} \nabla \cdot \left(e^{-F} a \nabla \cdot \right) + J \cdot \nabla = \mathcal{S} + \mathcal{A}, \\
J &= \frac{1}{2}e^F \nabla \cdot(e^{-F} a) - b,
\end{align}
see \cite{Pav14}. The vector field $J$ is called stationary probability flow. In the form \eqref{eq:generator_decomposition}, $-\mathcal{L}$ is a special case of an elliptic second order differential operator on $L^2_\mu$, given by
\begin{equation} \label{eq:def_second_order_op}
    \mathcal{T} = -\frac{1}{2}e^F \nabla \cdot \left( e^{-F} a \nabla \cdot \right) + J \cdot \nabla + W,
\end{equation}
for scalar functions $F,\, W$, a uniformly positive definite matrix field $a$, and a vector field $J$. 

\begin{remark}
Because of the general form of \eqref{eq:def_second_order_op}, we avoid making too many assumptions about the coefficients of $\mathcal{T}$ or its domain of definition. The goal is to derive numerical algorithms using a minimal set of assumptions. A detailed analysis of the interplay between domains and properties of the RKHS will be carried out in future publications.
\end{remark}

If $F \equiv 0$, we obtain generalized Schrödinger operators as another special case, i.e.,
\begin{equation}
\label{eq:def_schroedinger_op}
\mathcal{H} = -\frac{1}{2} \nabla\cdot \left( a\nabla \cdot \right) + J \cdot \nabla + W,
\end{equation}
with $W$ called the potential energy in quantum mechanics. In particular, with the reduced Planck constant $ \hbar $ and the mass $ \mathbf{m} $, setting $ a \equiv \frac{\hbar^2}{\mathbf{m}} I $ and $ J \equiv 0 $ leads to the Hamiltonian $ \mathcal{H} = - \frac{\hbar^2}{2 \mathbf{m}} \Delta + W$ of the time-independent Schrödinger equation in quantum mechanics:
\begin{equation} \label{eq:TISE}
    \mathcal{H} \psi = E \psi.
\end{equation}
We note for later use that, under certain conditions, Schrödinger operators and Koopman generators are equivalent, see, e.g., \cite[Chapter 4.9]{Pav14}. For the sake of completeness, the proof is shown in Appendix~\ref{sec:Proofs}.

\begin{lemma} \label{lem:KBE_SE}
The ergodic generator $-\mathcal{L}$ with unique positive invariant density $\rho_0 \propto \exp(-F)$ is unitarily equivalent to the Schrödinger operator $\mathcal{H}$ in \eqref{eq:def_schroedinger_op} on $L^2$, with $J$ remaining unchanged and $W$ given by:
\begin{align*}
W &= -\frac{1}{4} \nabla \cdot(a \nabla F) + \frac{1}{8} \nabla F^\top a \nabla F + \frac{1}{2} J \cdot \nabla F.
\end{align*}
The function $e^{-\frac{1}{2}F}$ is an eigenfunction of $\mathcal{H}$ with eigenvalue zero. Conversely, let $\mathcal{H}$ be as in \eqref{eq:def_schroedinger_op}, and assume there is a non-degenerate smallest eigenvalue $E_0$ with strictly positive real eigenfunction $\psi_0 = \exp(-\eta)$. Then $\mathcal{H}$ is unitarily equivalent to a negative ergodic generator $-\mathcal{L}$ on $L^2_\mu$, where $\rho_0 \propto \exp(-2\eta)$ is the density associated with $ \mu $, and $\rho_0$ is invariant for the corresponding SDE. The explicit form of $-\mathcal{L}$ is given by
\begin{equation*}
-\mathcal{L} = e^\eta\left[\mathcal{H} - E_0 \right] (e^{-\eta} \cdot) = -\frac{1}{2}e^{2\eta} \nabla \cdot \left(e^{-2\eta} a \nabla \cdot \right) +  J \cdot \nabla.
\end{equation*}
\end{lemma}

\begin{corollary} \label{corr:KBE_SE}
Applying Lemma~\ref{lem:KBE_SE} to \eqref{eq:TISE}, we have
\begin{align*}
    \frac{1}{\psi_0}(\mathcal{H} - E_0)(\psi_0 f)
        = - \left(-\frac{\hbar^2}{\mathbf{m}} \nabla \eta \cdot \nabla f + \frac{\hbar^2}{2 \mathbf{m}} \Delta f \right)
        = - \mathcal{L} f,
\end{align*}
where $ \mathcal{L} $ is the Koopman generator of a drift-diffusion process (see Remark \ref{rem:overdamped_langevin}) with potential (up to an additive constant)
\begin{equation*}
    V(x) = \frac{\hbar^2}{\mathbf{m}} \eta(x),
\end{equation*}
and temperature $ \beta^{-1} = \frac{\hbar^2}{2 \mathbf{m}} $.
\end{corollary}

We will exploit this duality below to apply methods developed for the Koopman operator or generator to the Schrödinger operator. More details on quantum chemistry in general and also the quantum harmonic oscillator and the hydrogen atom studied in Section~\ref{sec:Applications} can be found, e.g., in \cite{Levine2000}.

\subsection{Reproducing kernel Hilbert spaces and derivative reproducing properties}

We aim at representing the differential operators introduced above in reproducing kernel Hilbert spaces.

\begin{definition}
Let $ \mathbb{X} $ be a set and $ \mathbb{H} $ a space of functions $ f \colon \mathbb{X} \to \R $. Then $ \mathbb{H} $ is called a \emph{reproducing kernel Hilbert space} (RKHS) with inner product $ \innerprod{\cdot}{\cdot}_\mathbb{H} $ if a function $ k \colon \mathbb{X} \times \mathbb{X} \to \R $ exists such that
\begin{enumerate}[label=(\roman*), itemsep=0ex, topsep=1ex]
\item $ \innerprod{f}{k(x, \cdot)}_\mathbb{H} = f(x) $ for all $ f \in \mathbb{H} $, and
\item $ \mathbb{H} = \overline{\mspan\{k(x, \cdot) \mid x \in \mathbb{X} \}} $.
\end{enumerate}
\end{definition}

The function $ k $ is called \emph{kernel}. It was shown that every RKHS has a unique symmetric positive definite\footnote{Here, we use the terms \emph{positive definite} and \emph{strictly positive definite}, i.e., positive definite means that $ \sum_{r=1}^M \sum_{s=1}^M \gamma_r \gamma_s \ts k(x_r, x_s) \ge 0 $ for all $ M \in \mathbb{N} $, $ \gamma_1, \dots, \gamma_M \in \R $, and $ x_1, \dots, x_M \in \mathbb{X} $.} reproducing kernel and that, conversely, every symmetric positive definite kernel spans a unique RKHS, see \cite{Aronszajn50,Schoe01, Steinwart2008:SVM}. Frequently used kernels include the polynomial kernel and the Gaussian kernel, given by
\begin{equation*}
    k(x, x^\prime) = (c + x^\top x^\prime)^q
    \quad \text{and} \quad
    k(x, x^\prime) = \exp\left(-\frac{\norm{x-x^\prime}^2}{2 \ts \varsigma^2}\right),
\end{equation*}
respectively. Here, $ q \in \mathbb{N} $ is the degree of the polynomial kernel, $ c \ge 0 $ a parameter, and $ \varsigma $ the bandwidth of the Gaussian kernel. We now introduce the partial derivative reproducing properties of RKHSs~\cite{Zhou08}. Let $ \alpha = (\alpha_1, \dots, \alpha_d) \in \mathbb{N}_0^d $ be a multi-index and $ \abs{\alpha} = \sum_{i=1}^d \alpha_i $. Furthermore, for a fixed $ p \in \mathbb{N}_0 $, we define the index set $ I_p = \{ \alpha \in \mathbb{N}_0^d: \abs{\alpha} \le p \} $. Given $ f \colon \mathbb{X} \to \R $, let $ D^\alpha $ denote the partial derivative---assuming it exists---
\begin{equation*}
    D^\alpha f = \pd{^{\abs{\alpha}}}{x_1^{\alpha_1} \dots \ts \partial x_d^{\alpha_d}} f.
\end{equation*}
Thus, the $ i $th entry of the gradient is given by $ D^{e_i} f $ and the $ (i,j)$th entry of the Hessian by $ D^{e_i + e_j} $, where $ e_i $ and $ e_j $ are the $ i $th and $ j $th unit vector, respectively. When we apply the differential operator $ D^\alpha $ to the kernel $ k $, the multi-index $ \alpha $ is assumed to be embedded into $ \mathbb{N}_0^{2 \ts d} $ by adding zeros, i.e., the derivatives are computed with respect to the first argument of the kernel. Also when we write $ \nabla k(x, x^\prime) $, the gradient is computed with respect to $ x $. In what follows, let $ k(x, \cdot) = \phi(x) $, where $ \phi $ is the canonical feature space mapping.

\begin{theorem}[\cite{Zhou08}]
Given $ p \in \mathbb{N}_0 $ and a positive definite kernel $ k \colon \mathbb{X} \times \mathbb{X} \to \R $ with $ k \in C^{2 \ts p}(\mathbb{X} \times \mathbb{X}) $, the following holds:
\begin{enumerate}[label=(\roman*), itemsep=0ex, topsep=1ex]
\item $ D^\alpha k(x, \cdot) \in \mathbb{H} $ for any $ x \in \mathbb{X} $ and $ \alpha \in I_p $.
\item $ (D^\alpha f)(x) = \innerprod{D^\alpha k(x, \cdot)}{f}_{\mathbb{H}} $ for any $ x \in \mathbb{X} $, $ f \in \mathbb{H} $, and $ \alpha \in I_p $.
\end{enumerate}
\end{theorem}

The second property is called \emph{derivative reproducing property}. For $ p = 0 $, this reduces to the standard reproducing property of RKHSs.

\begin{example} Let us consider the two aforementioned kernels:
\begin{enumerate}[wide, itemindent=\parindent, itemsep=0ex, topsep=0.5ex]
\item For the polynomial kernel, we obtain
\begin{equation*}
    D^{e_i} k(x, x^\prime) = q \ts x^\prime_i (c + x^\top x^\prime)^{q-1}
    \quad \text{and} \quad
    D^{e_i + e_j} k(x, x^\prime) = q \ts (q-1) \ts x^\prime_i \ts x^\prime_j (c + x^\top x^\prime)^{q-2}.
\end{equation*}
\hspace*{\itemindent plus 2ex} Thus, $ \nabla \ts k(x, x^\prime) = q \ts x^\prime \ts (c + x^\top x^\prime)^{q-1} $ and $ \nabla^2 \ts k(x, x^\prime) = q \ts (q-1) \ts x^\prime \ts x^{\prime\top} (c + x^\top x^\prime)^{q-2} $.
\item Similarly, for the Gaussian kernel, this results in
\begin{align*}
    D^{e_i} k(x, x^\prime) &= -\frac{1}{\varsigma^2} (x_i - x^\prime_i) \ts k(x, x^\prime), \\
    D^{e_i + e_j} k(x, x^\prime) &=
    \begin{cases}
        \bigg[\dfrac{1}{\varsigma^4}(x_i - x^\prime_i)^2 - \dfrac{1}{\varsigma^2}\bigg] k(x, x^\prime), & i = j, \\[2ex]
        \dfrac{1}{\varsigma^4}(x_i - x^\prime_i) (x_j - x^\prime_j) \ts k(x, x^\prime), & i \ne j,
    \end{cases}
\end{align*}
\hspace*{\itemindent plus 2ex} $ \nabla \ts k(x, x^\prime) = -\frac{1}{\varsigma^2} (x - x^\prime) \ts k(x, x^\prime) $, and $ \nabla^2 \ts k(x, x^\prime) = \Big[\frac{1}{\varsigma^4}(x - x^\prime)(x - x^\prime)^\top - \frac{1}{\varsigma^2} I \Big] k(x, x^\prime) $. \exampleSymbol
\end{enumerate}
\end{example}

For the numerical experiments below, we will mainly use the Gaussian kernel\footnote{To get error estimates, it might be more convenient to use Wendland functions \cite{wendland_2004}. We leave the formal analysis of the methods developed in this paper for future work.}.

\section{Kernel-based representation of differential operators}
\label{sec:kgEDMD}

In this section, we introduce the Galerkin projection of the differential operators discussed above onto the RKHS, including the Koopman generator and Schrödinger operator. We then move on to show how these projected operators can be estimated from data.

\subsection{Galerkin projection of operators}

Let $\mu$ denote a probability measure on the state space $\mathbb{X}$, with density $\rho_0 \propto e^{-F}$ for a generalized potential $F$.

\begin{definition}
\label{def:rkhs_operators}
We define the \emph{covariance operator} $ \mathcal{C}_{00} \colon \mathbb{H} \to \mathbb{H} $ by
\begin{equation}
\label{eq:def_cov_op}
    \mathcal{C}_{00} = \int \phi(x) \otimes \phi(x) \ts \mathrm{d} \mu(x),
\end{equation}
and an operator $ \mathcal{T}_{\mathbb{H}} \colon \mathbb{H} \to \mathbb{H} $ by
\begin{equation}
\label{eq:def_rkhs_op}
\begin{split}
    \mathcal{T}_{\mathbb{H}} &= \int \phi(x) \otimes \left[ - \frac{1}{2}\sum_{i=1}^d \sum_{j=1}^d a_{ij}(x) D^{e_i + e_j} \phi(x) \right] \ts \mathrm{d} \mu(x) \\
    & \quad + \int \phi(x) \otimes \left[\sum_{i=1}^d \left(J_i(x) - \frac{1}{2} e^{F(x)} \nabla \cdot (e^{-F(x)}a_{:, i}(x))\right) D^{e_i}\phi(x) \right] \ts \mathrm{d} \mu(x) \\
    & \quad + \int W(x) \phi(x) \otimes \phi(x)  \ts \mathrm{d} \mu(x).
\end{split}
\end{equation}
If $J \equiv 0$, we define $\mathcal{T}_\mathbb{H}$ by
\begin{equation}
\label{eq:def_rkhs_sym}
    \mathcal{T}_{\mathbb{H}} = \int \left[\frac{1}{2} \sum_{i=1}^d \sum_{j=1}^d a_{ij}(x) \left(D^{e_i} \phi(x) \otimes D^{e_j} \phi(x)\right)\right] + W(x) \phi(x)\otimes \phi(x)  \,\mathrm{d} \mu(x).
\end{equation}
\end{definition}

The operator $ \mathcal{C}_{00} $ is the standard covariance operator $ \mathcal{C}_{\scriptscriptstyle X\!X} $, see \cite{Baker70, Baker73}, the operator $ \mathcal{T}_{\mathbb{H}} $ mimics the action of the bilinear form $\innerprod{\mathcal{T}f}{g}_\mu$ on the RKHS. It plays the same role as the cross-covariance operator $ \mathcal{C}_{\scriptscriptstyle X\!Y} $ for the Koopman operator in \cite{KSM19}. The form of the symmetric operator for $J \equiv 0$ is motivated by the symmetry of $\mathcal{T}$, and that, at least formally
\begin{equation*}
\innerprod{\mathcal{T}f}{g}_\mu = \int \left[\frac{1}{2} \nabla f(x)^\top a(x) \nabla g(x) \right] + W(x) f(x) g(x) \,\mathrm{d} \mu(x),
\end{equation*}
see also \cite{DAVIES1996}.

\begin{lemma} \label{lem:RKHS Galerkin}
Assume that $\mathbb{H} \subset \mathcal{D}(\mathcal{T})$, and that all terms appearing under the integral signs in \eqref{eq:def_cov_op} and \eqref{eq:def_rkhs_op} (or \eqref{eq:def_rkhs_sym}) are in $L^1_\mu$ as bounded operators on $\mathbb{H}$, that is
\begin{align}
\label{eq:l1_conditions_th_1}
& \int |a_{ij}(x)| \|D^{e_i + e_j} \phi(x)\|_{\mathbb{H}} \|\phi(x)\|_\mathbb{H}  \ts \mathrm{d} \mu(x) < \infty, \\
\label{eq:l1_conditions_th_2}
& \int \left(|J_i(x)| + \frac{1}{2} e^{F(x)} |\nabla \cdot (e^{-F(x)}a_{:, i}(x))| \right) \|D^{e_i}\phi(x)\|_\mathbb{H} \|\phi(x)\|_\mathbb{H} \ts \mathrm{d} \mu(x) < \infty,  \\
\label{eq:l1_conditions_th_3}
&\int  |W(x)| \|\phi(x)\|_\mathbb{H} \|\phi(x)\|_\mathbb{H}  \ts \mathrm{d} \mu(x) < \infty, \\
\label{eq:l1_conditions_th_4}
&\int \|\phi(x)\|_\mathbb{H} \|\phi(x)\|_\mathbb{H}  \ts \mathrm{d} \mu(x) < \infty.
\end{align}
Then, for all $ f, g \in \mathbb{H} $, 
\begin{align*}
\innerprod{\mathcal{T}f}{g}_\mu &= \innerprod{\mathcal{T}_\mathbb{H}f}{g}_\mathbb{H}, & \innerprod{f}{g}_\mu &= \innerprod{\mathcal{C}_{00}f}{g}_\mathbb{H}.
\end{align*}
\end{lemma}

The proof can be found in Appendix~\ref{sec:Proofs}. It uses the derivative reproducing properties and the definition of rank-one operators. Note that
\begin{align*}
    D^{e_i} f(x) \ts g(x)
        &= \innerprod{D^{e_i} \phi(x)}{f}_\mathbb{H} \innerprod{\phi(x)}{g}_\mathbb{H} \\
        &= \innerprod{D^{e_i} \phi(x) \otimes \phi(x)}{f \otimes g}_{\mathbb{H}\otimes \mathbb{H}} \\
        &= \innerprod{(\phi(x) \otimes D^{e_i} \phi(x)) f}{g}_\mathbb{H}.
\end{align*}

\begin{lemma}
\label{lem:rep_L_H}
Assume that $ \mathcal{T} f \in \mathbb{H} $ for all $ f \in \mathbb{H} $, then $ \mathcal{T}_{\mathbb{H}} f = \mathcal{C}_{00} \mathcal{T} f $.
\end{lemma}

\begin{proof}
The proof is similar to the one for the corresponding result for kernel transfer operators, see \cite{KSM19}. With the previous lemma, we obtain
\begin{align*}
    \innerprod{\mathcal{C}_{00} \mathcal{T} f}{g}_\mathbb{H} &= \mathbb{E}^\mu[ (\mathcal{T} f)(x) \ts g(x) ] \\
    &= \int (\mathcal{T}f)(x) g(x) \ts \mathrm{d} \mu(x) \\
    &= \innerprod{\mathcal{T}_\mathbb{H} f}{g}_\mathbb{H}
\end{align*}
for arbitrary $ g \in \mathbb{H} $.
\end{proof}

If the assumptions of Lemma \ref{lem:rep_L_H} are satisfied and the operator $\mathcal{C}_{00}$ is invertible, the RKHS operators defined above can be used to compute exact eigenfunctions of $\mathcal{T}$. Indeed, if $\varphi$ is a solution of
\begin{equation*}
    \mathcal{T}_{\mathbb{H}} \varphi = \mathcal{C}_{00} \mathcal{T} \varphi = \lambda \ts \mathcal{C}_{00} \varphi,
\end{equation*}
then multiplying this equation by $\mathcal{C}_{00}^{-1}$ shows that $\varphi$ is also an eigenfunction for $\mathcal{T}$. A typical approach to circumvent the potential nonexistence of the inverse of the covariance operator is to consider a regularized version $ \mathcal{T}_\varepsilon = (\mathcal{C}_{00} + \varepsilon I)^{-1} \mathcal{T}_\mathbb{H} $ for a regularization parameter $ \varepsilon $. However, the assumptions of Lemma \ref{lem:rep_L_H} are strong and may be hard to verify in practice. But in any case, Lemma \ref{lem:RKHS Galerkin} shows that the operators defined in Definition~\ref{def:rkhs_operators} provide a Galerkin approximation of the full operator in the RKHS $\mathbb{H}$.

\subsection{Empirical estimates}
\label{subsec:empirical_estimates}

The next step is to derive empirical estimates of the operators defined above. Given training data $ \{\ts x_m\ts\}_{m=1}^M $, sampling the probability distribution $ \mu $, we define $ \Phi = [\phi(x_1), \dots, \phi(x_M)] $ and $ \mathrm{d}\Phi = [\mathrm{d}\phi(x_1), \dots, \mathrm{d}\phi(x_M)] $, where
\begin{align*}
    \mathrm{d}\phi(x_m) &= -\frac{1}{2} \sum_{i=1}^d \sum_{j=1}^d a_{ij}(x_m) D^{e_i + e_j} \phi(x_m) \\
    &\quad + \sum_{i=1}^d \left[J_i(x_m) - \frac{1}{2} \sum_{j=1}^d e^{F(x_m)} \frac{\partial}{\partial x_j}(e^{-F(x_m)}a_{ji}(x_m))\right] D^{e_i}\phi(x_m) \\ &\quad+ W(x_m) \phi(x_m).
\end{align*}
If $\mathcal{T}$ is the generator of an SDE with invariant measure $\mu$, the data can also be obtained by integrating the stochastic dynamics with initial condition drawn from $\mu$. We see that $ \Phi $ is the standard feature map and $ \mathrm{d}\Phi $ contains the action of the differential operator. The empirical estimates of the operators $ \mathcal{C}_{00} $ and $\mathcal{T}_{\mathbb{H}}$ are then given by the following expressions:
\begin{alignat*}{4}
    \widehat{\mathcal{C}}_{00}
        &= \frac{1}{M} \sum_{m=1}^M \phi(x_m) \otimes \phi(x_m)
        &&= \frac{1}{M} \Phi \Phi^\top, \\
    \widehat{\mathcal{T}}_{\mathbb{H}}
        &= \frac{1}{M} \sum_{m=1}^M \phi(x_m) \otimes \mathrm{d}\phi(x_m)
        &&= \frac{1}{M} \Phi \ts \mathrm{d}\Phi^\top.
\end{alignat*}
Note that these are still finite-rank operators on the full RKHS $\mathbb{H}$. For the symmetric RKHS operator $\mathcal{T}_\mathbb{H}$, we need to define the empirical estimate in a slightly different way. Decompose the positive definite matrix $a(x_m) = \sigma(x_m) \sigma(x_m)^\top$. With
\begin{equation*}
    \mathrm{d}\phi_l(x_m) = \sum_{i=1}^d \sigma_{il}(x_m) D^{e_i}\phi(x_m)= \left[\nabla \phi(x_m)^\top \sigma_l(x_m)\right],
\end{equation*}
where $\sigma_l$ is the $l$th column of $\sigma$, the empirical RKHS operator becomes
\begin{equation*}
    \widehat{\mathcal{T}}_{\mathbb{H}} = \frac{1}{2M}\sum_{m=1}^M \sum_{l=1}^d \mathrm{d}\phi_l(x_m) \otimes \mathrm{d}\phi_l(x_m) + \frac{1}{M}\sum_{m=1}^M W(x_m) \phi(x_m)\otimes \phi(x_m).
\end{equation*}

\begin{remark}
If the feature space associated with the kernel $ k $ is finite-dimensional and known explicitly, i.e., $ \phi(x) = [\phi_1(x), \dots, \phi_N(x)]^\top $ and $ k(x, x^\prime) = \innerprod{\phi(x)}{\phi(x^\prime)} $, then for the Koopman generator we obtain gEDMD as a special case, with $ \widehat{\mathcal{C}}_{00} = \widehat{G} $ and $ \widehat{\mathcal{T}}_{\mathbb{H}} = -\widehat{A}^{\ts\top} $. However, the goal is to rewrite gEDMD in such a way that only kernel evaluations are required since $ \phi $ can potentially be infinite-dimensional and might only be defined implicitly.
\end{remark}

\subsection{Weak formulation and numerical algorithm}

With Lemma \ref{lem:RKHS Galerkin} in mind, we now proceed to the weak formulation of the eigenvalue problem for the operator $\mathcal{T}$. We then define the quadratic forms
\begin{alignat*}{4}
\mathcal{Q}(f, g) &= \innerprod{\mathcal{T}f}{g}_\mu, && \quad f,g\in \mathcal{D}_\mathcal{Q}, &  \mathcal{S}(f, g) &= \innerprod{f}{g}_\mu, && \quad f,g \in L^2_\mu, \\
\mathcal{Q}_\mathbb{H}(f, g) &= \innerprod{\mathcal{T}_\mathbb{H}f}{g}_\mathbb{H}, && \quad f,g \in \mathbb{H}, & \mathcal{S}_\mathbb{H}(f, g) &= \innerprod{\mathcal{C}_{00}f}{g}_\mathbb{H}, && \quad f,g \in \mathbb{H}, \\
\widehat{\mathcal{Q}}_\mathbb{H}(f, g) &= \innerprod{\widehat{\mathcal{T}}_\mathbb{H}f}{g}_\mathbb{H}, && \quad f,g \in \mathbb{H}, & \qquad \widehat{\mathcal{S}}_\mathbb{H}(f, g) &= \innerprod{\widehat{\mathcal{C}}_{00}f}{g}_\mathbb{H}, && \quad f,g \in \mathbb{H},
\end{alignat*}
where $\mathcal{D}_\mathcal{Q}$ is the domain of the quadratic form $\mathcal{Q}$. We consider the weak eigenvalue problems
\begin{alignat}{2}
\label{eq:weak_ev_generator}
\mathcal{Q}(f_n, g) &= \lambda_n \mathcal{S}(f_n, g) && \quad \forall g \in \mathcal{D}_\mathcal{Q}, \\
\label{eq:weak_ev_rkhs}
\mathcal{Q}_\mathbb{H}(\widetilde{f}_n, g) &= \widetilde{\lambda}_n \mathcal{S}_\mathbb{H}(\widetilde{f}_n, g) && \quad \forall g\in \mathbb{H}, \\
\label{eq:weak_ev_data}
\widehat{\mathcal{Q}}_\mathbb{H}(\widehat{f}_n, g) &= \widehat{\lambda}_n \widehat{\mathcal{S}}_\mathbb{H}(\widehat{f}_n, g) && \quad \forall g\in \mathbb{H}.
\end{alignat}
We will now rewrite \eqref{eq:weak_ev_data} in such a way that only kernel evaluations---in the form of Gram matrices---are required. The derivation is similar to the kernel transfer operator counterpart in \cite{KSM19}, but we now need to consider derivatives at the training data points instead of the time-lagged variables. We start by restricting \eqref{eq:weak_ev_data} to the finite-dimensional space $ \mathbb{H}^M = \mathrm{span}\{\phi(x_m) \}_{m=1}^M $, which we assume to be $ M $-dimensional. Elements of this space are of the form $ f = \Phi \ts u $  for some vector $ u \in \R^m $. We examine the quadratic forms $ \widehat{\mathcal{Q}}_\mathbb{H} $ and $ \widehat{\mathcal{S}}_\mathbb{H} $ on this space.

\begin{lemma} \label{lem:EVP}
A solution of the problem $ \widehat{\mathcal{Q}}_\mathbb{H}(f, g) = \widehat{\lambda} \ts \widehat{\mathcal{S}}_\mathbb{H}(f, g) $ is given by $ f = \Phi \ts u $, where $ u $ is a solution of one of the following generalized eigenvalue problems:
\begin{enumerate}[label=(\roman*), itemsep=0ex, topsep=1ex]
\item In the general case, $ u $ solves $ G_{2} \ts u = \widehat{\lambda} \ts G_{0} \ts u $, where the entries of the matrices $ G_2 $ and $ G_0 $ are given by
\begin{align*}
    \big[G_{2}\big]_{mr} &= \left[\mathrm{d}\phi(x_m)\right](x_r), & \big[G_{0}\big]_{mr} &= \left[\phi(x_m)\right](x_r).
\end{align*}
\item Analogously, for the symmetric case, we obtain $ \frac{1}{2} \sum_{l=1}^d G_{1}^{(l)} \ts G_{1}^{(l)} \ts u = \widehat{\lambda} \ts G_{0} \ts G_{0} \ts u $, where we define
\begin{equation*}
    \Big[G_{1}^{(l)} \ts \Big]_{mr} = \sigma_{l}(x_m)^\top \nabla k(x_m, x_r)
\end{equation*}
and $ \sigma_{l}(x_m) $ is the $l$th column of the matrix $ \sigma(x_m) $.
\end{enumerate}
\end{lemma}

The proofs are shown in Appendix~\ref{sec:Proofs}. Since $ \left[\phi(x_m)\right](x_r) = k(x_m, x_r) $, $ G_0 $ is the standard Gram matrix. The reversible case requires only first-order derivatives of the kernel. Furthermore, only trajectory data sampled from the invariant distribution $ \mu $ and estimates of the diffusion term $ \sigma $ are needed. For typical problems, $ \sigma $ is constant and not position-dependent. As a result, the diffusion term needs to be estimated only once or might even be known. For molecular dynamics problems, for instance, it is proportional to the square root of the temperature. The overall approach is summarized in the following algorithm. Note that it is not a direct kernelization of gEDMD, but an extension that approximates the Koopman generator as a special case.

\begin{textalgorithm} \label{alg:kgEDMD}
The final numerical algorithm can be summarized as follows:
\begin{enumerate}[itemsep=0ex, topsep=0.5ex]
\item Choose a kernel $ k $ and compute all its required derivatives, either analytically or with the aid of automatic differentiation.
\item Assemble the Gram matrices $ G_{2} $ and $ G_{0} $ or, if the system is symmetric, $ G_{1}^{(l)} $, for $ l = 1, \dots, d $, and $ G_{0} $.
\item Solve the corresponding eigenvalue problem described in Lemma~\ref{lem:EVP} to obtain an eigenvector~$ u $.
\item An eigenfunction is then given by $ \varphi = \Phi \ts u $.
\end{enumerate}
\end{textalgorithm}

The two main steps of the algorithm are assembling the Gram matrices and solving the generalized eigenvalue problem. Since the size of the eigenvalue problem depends on the number of data points, the cost is cubic in $ M $. This is a drawback of many kernel-based methods. The efficient approximation of solutions to this eigenvalue problem for large data sets will be considered in future work.

\subsection{Analysis}

In this section, we provide some preliminary analysis of the methods introduced above. The first result concerns the convergence of the empirical estimates.

\begin{lemma}
As $M \rightarrow \infty$, the empirical estimates defined in Section \ref{subsec:empirical_estimates} converge to the corresponding RKHS operators in Definition~\ref{def:rkhs_operators} with respect to the operator norm for almost all data sequences $ \{ \ts x_m \ts \}_{m=1}^M$, if the data were generated either as i.i.d.\ samples from $ \mu $, or by integrating a stochastic dynamics which is ergodic with respect to $ \mu $.
\end{lemma}

\begin{proof}
The statement follows from ergodicity of the underlying dynamics, the integrability conditions in Lemma~\ref{lem:RKHS Galerkin}, and the Birkhoff individual ergodic theorem for Banach space valued functions \cite{CHACON1962}.
\end{proof}

Next, we generalize \cite[Theorem 7]{RBD10} to obtain convergence rates on the empirical estimates for i.i.d.\ data:

\begin{lemma} \label{lem:cov_conc}
Assume that (\ref{eq:l1_conditions_th_1}-- \ref{eq:l1_conditions_th_4}) hold. Then:
\begin{itemize}\itemsep 0pt
\item[(i)] The operators $C_{00}$, $\widehat{C}_{00}$, $\mathcal{T}_\mathbb{H}$ and $\hat{\mathcal{T}}_\mathbb{H}$ are Hilbert--Schmidt.
\item[(ii)] Let $\delta\in(0,1]$. Assume the coefficients of the operator $\mathcal{T}$ are all globally bounded, and let $\sup_{x\in \mathbb{X}} D^\alpha k(x, x) < \infty $ for all $|\alpha| \leq 4$ ($|\alpha| \leq 2$ in the symmetric case). If the data are drawn i.i.d.\ from the distribution $\mu$, then there are constants $\kappa_0, \kappa_1$ such that with probability at least $1-\delta$,
\begin{align*}
\| \mathcal{C}_{00}-\widehat{C}_{00}\|_{HS} &\leq \frac{2\kappa_0 \sqrt{2}}{\sqrt{M}}\log^{1/2}\frac{2}{\delta}, &
\| \mathcal{T}_{\mathbb{H}}-\widehat{\mathcal{T}}_{\mathbb{H}}\|_{HS} &\leq \frac{2\kappa_1 \sqrt{2}}{\sqrt{M}}\log^{1/2}\frac{2}{\delta},
\end{align*}
where the $\norm{\cdot}_{\mbox{HS}}$ is the Hilbert--Schmidt norm.
\end{itemize}
\end{lemma}
\begin{proof}
(i) The empirical estimates are all finite rank and therefore Hilbert-Schmidt. For $\mathcal{C}_{00}$ and $\mathcal{T}_\mathbb{H}$, this follows from the integrability conditions and the first part of the proof of Lemma \ref{lem:RKHS Galerkin}, see Appendix \ref{sec:Proofs}.
\newline
(ii) For $\mathcal{C}_{00}$, the bound was already proven in \cite{RBD10} with $\kappa_0 = \sup_{x\in \mathbb{X}} k(x, x)^2$. We can employ the same strategy to obtain the bound for $\mathcal{T}_\mathbb{H}$. Consider the operator $\hat{\mathcal{T}}^m_\mathbb{H} = \phi(x_m) \otimes \mathrm{d}\phi(x_m) - \mathcal{T}_\mathbb{H}$, which satisfies $\mathbb{E}^\mu[\hat{\mathcal{T}}^m_\mathbb{H}] = 0$. By global boundedness of the coefficients of $\mathcal{T}$ and by
\begin{align*}
\|\phi(x) \otimes D^\alpha \phi(x)\|_{HS} &= \|\phi(x)\|_\mathbb{H} \|D^\alpha \phi(x)\|_{\mathbb{H}} \\
&= \innerprod{k(x, \cdot)}{k(x, \cdot)}_\mathbb{H}^{1/2} \innerprod{D^\alpha k(x, \cdot)}{D^\alpha k(x, \cdot)}_\mathbb{H}^{1/2} \\
&= k(x, x)^{1/2} D^{2\alpha} k(x, x)^{1/2},
\end{align*}
we can find a $\kappa_1$ such that $\| \phi(x) \otimes \mathrm{d}\phi(x) \|_{HS} \leq \kappa_1$ for all $x\in \mathbb{X}$. We then have $\| \hat{\mathcal{T}}^m_\mathbb{H} \|_{HS} \leq 2\kappa_1$, and the result follows from the concentration bound \cite[Equation 3]{RBD10}.
\end{proof}

Finally, we show that solutions of \eqref{eq:weak_ev_rkhs} are also eigenvalues of the full operator $\mathcal{T}$ if the RKHS is dense in $\mathcal{D}_\mathcal{Q}$:
\begin{proposition}
Let $\mathbb{H}$ be dense in $\mathcal{D}_\mathcal{Q}$ with respect to the norm in $L^2_\mu$. If $\widetilde{\psi}_\ell \in \mathbb{H}$ is an eigenfunction of \eqref{eq:weak_ev_rkhs}, it is also an eigenfunction of $\mathcal{T}$ with the same eigenvalue.
\end{proposition}
\begin{proof}
Let $\widetilde{\psi}_\ell$ solve the variational problem \eqref{eq:weak_ev_rkhs}. The definition of the operators $\mathcal{C}_{00},\, \mathcal{T}_\mathbb{H}$ implies that for all $\phi \in \mathbb{H}$:
\begin{align*}
\innerprod{\mathcal{T} \widetilde{\psi}_\ell}{\phi}_\mu &= \innerprod{\mathcal{T}_\mathbb{H} \widetilde{\psi}_\ell}{\phi}_{\mathbb{H}} = \widetilde{\lambda}_\ell \innerprod{\mathcal{C}_{00} \widetilde{\psi}_\ell}{\phi}_{\mathbb{H}} = \widetilde{\lambda}_\ell \innerprod{\widetilde{\psi}_\ell}{\phi}_\mu.
\end{align*}
By density of the RKHS, this also holds for all $\phi \in \mathcal{D}_\mathcal{Q}$, and consequently, $\widetilde{\psi}_\ell$ is an eigenfunction of $\mathcal{T}$.
\end{proof}

Note that even if the RKHS is dense, there might be additional eigenfunctions which are not contained in $\mathbb{H}$ and which will not appear as solutions of \eqref{eq:weak_ev_rkhs}.

\section{Applications}
\label{sec:Applications}

The methods described above have important applications in molecular dynamics and quantum physics, which we will show in an exemplary way, but can in principle be applied to data generated by arbitrary dynamical systems and also other differential operators. The code and select examples are available at \url{https://github.com/sklus/d3s/}. Note that this is just a proof-of-concept implementation and that the methods could be sped up significantly by vectorizing and parallelizing the code and by tailoring the implementation to specific kernels.

\subsection{Molecular dynamics}

Eigenvalues and eigenfunctions of transfer operators associated with molecular dynamics problems are often used to understand protein folding or binding/unbinding processes and their implied time scales. Conformations correspond to metastable sets and transitions between different conformations to crossing energy barriers. The slowest dynamical processes are encoded in eigenfunctions whose eigenvalues are close to zero. Large-scale molecular dynamics examples, analyzed using kernel EDMD, can also be found in \cite{KBSS18}. In this paper, we want to focus more on new applications.

\begin{example}

\begin{figure}
    \centering
    \begin{minipage}{0.32\textwidth}
        \centering
        \subfiguretitle{(a)}
        \includegraphics[width=\textwidth]{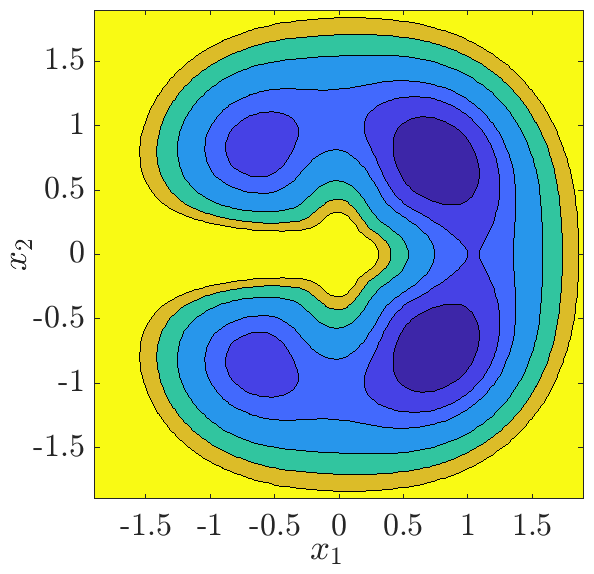}
    \end{minipage}
    \begin{minipage}{0.32\textwidth}
        \centering
        \subfiguretitle{(b)}
        \includegraphics[width=\textwidth]{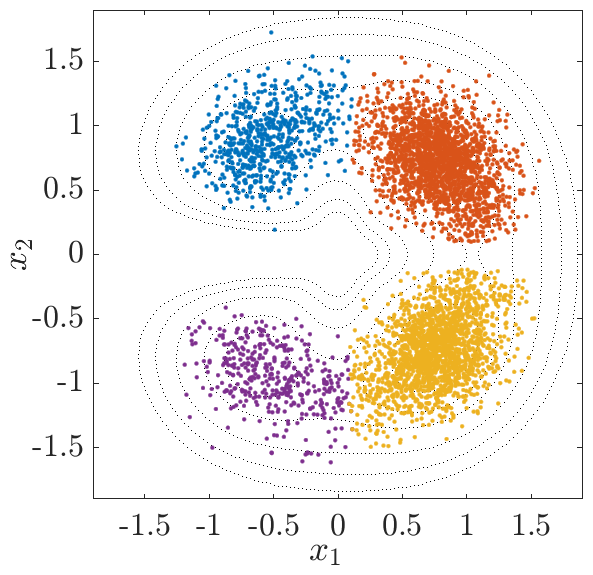}
    \end{minipage}
    \begin{minipage}{0.315\textwidth}
        \centering
        \subfiguretitle{(c)}
        \includegraphics[width=\textwidth]{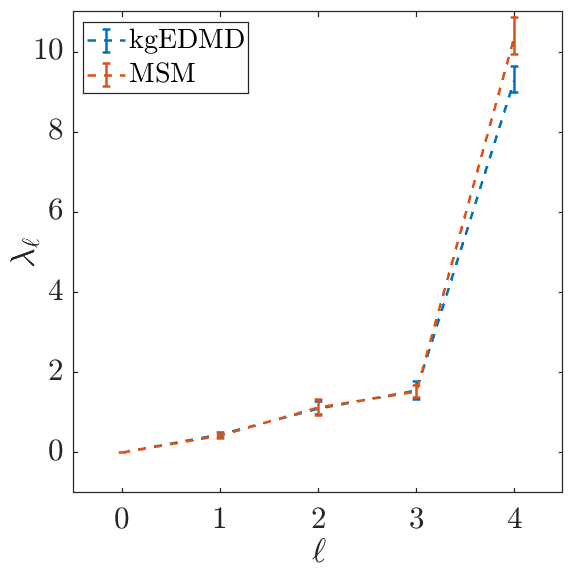}
    \end{minipage}
    \caption{(a) Quadruple-well potential. The color blue corresponds to small values and yellow to large values. (b) Clustering into four metastable sets based on SEBA. (c) Eigenvalues computed using kernel gEDMD and a  Markov state model. The bars indicate the estimated standard deviation.}
    \label{fig:LemonSlice}
\end{figure}

Let us consider the simple quadruple-well problem whose potential $ V $ is visualized in Figure~\ref{fig:LemonSlice}(a), see also \cite{KNPNCS20}. We first generate an equilibrated trajectory so that the training data set of size $ M = 5000 $ is sampled from the invariant distribution and then apply kernel gEDMD for reversible processes, choosing a Gaussian kernel with bandwidth $ \varsigma = 0.5 $. The operator $ -\mathcal{L} $ has four dominant eigenvalues $ \lambda_0 = 0.009 $, $ \lambda_1 = 0.400 $, $ \lambda_2 = 1.011 $, and $ \lambda_3 = 1.55 $, followed by a spectral gap. We then apply SEBA (\emph{sparse eigenbasis approximation}, see~\cite{FROYLAND19}) to cluster the dominant eigenfunctions into four metastable sets. The results are shown in Figure~\ref{fig:LemonSlice}(b). As expected, the sets correspond to the wells of the potential. The computation and clustering of the eigenfunctions takes approximately 4 minutes on a standard laptop (8 cores, 1.80 GHz, 16 GB of RAM). For comparison, we estimate the generator eigenvalues using a Markov state model. Applying both methods to 20 different trajectories, we compute the average of the eigenvalues and the standard deviation, see Figure~\ref{fig:LemonSlice}(c). The results are in excellent agreement. Clearly, the standard deviation increases for higher eigenvalues. \exampleSymbol

\end{example}

\subsection{Quantum mechanics}

The goal now is to apply data-driven methods to simple quantum mechanics problems of the form \eqref{eq:TISE} with $ \mathcal{H} = - \frac{\hbar^2}{2 \mathbf{m}} \Delta + W$.

\subsubsection{Generator EDMD for the Schrödinger equation}

Let us consider two systems for which the eigenfunctions are well-known.

\begin{example}
For the quantum harmonic oscillator with angular frequency $ \omega $, the potential can be written as $ W(x) = \frac{1}{2} \textbf{m} \ts \omega^2 x^2 $. The eigenfunctions $ \psi_\ell $ and corresponding energy levels $ E_\ell $ of this system can be computed analytically, we obtain
\begin{equation*}
    \psi_\ell(x) = \frac{1}{\sqrt{2^\ell\,\ell!}} \left(\frac{\mathbf{m} \omega}{\pi \hbar}\right)^{1/4} \exp\left(-\frac{\mathbf{m} \omega}{2 \hbar} x^2\right) H_\ell \left(\sqrt{\frac{\mathbf{m} \omega}{\hbar}} x \right)
\end{equation*}
and $ E_\ell = \hbar \ts \omega \left(\ell + \frac{1}{2}\right) $, for $ \ell = 0, 1, 2, \dots $. Here, $ H_\ell $ denotes the $\ell$th physicists' Hermite polynomial. For the numerical experiments, we set $ \hbar = \mathbf{m} = \omega = 1 $. Also the bandwidth of the kernel is set to $ \varsigma = 1 $. Computing the Gram matrices $ G_{2} $ and $ G_{0} $ for $ 100 $ uniformly distributed points in $ [-5, 5] $ and solving the corresponding eigenvalue problem, this results in the eigenfunctions shown in Figure~\ref{fig:QHO}. The probability densities $ p_\ell $ are defined by $ p_\ell(x) = \abs{\psi_\ell(x)}^2 $. \exampleSymbol

\begin{figure}
    \centering
    \begin{minipage}{0.49\textwidth}
        \centering
        \subfiguretitle{(a)}
        \includegraphics[width=0.9\textwidth]{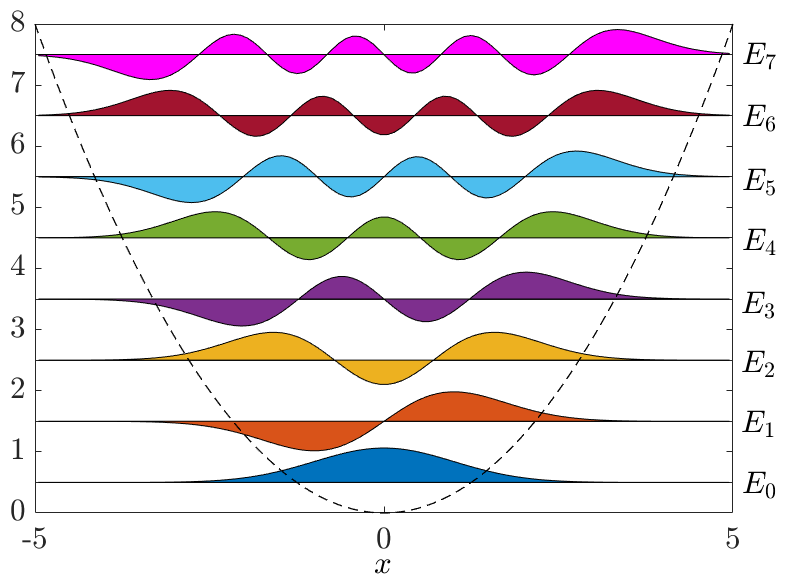}
    \end{minipage}
    \begin{minipage}{0.49\textwidth}
        \centering
        \subfiguretitle{(b)}
        \includegraphics[width=0.9\textwidth]{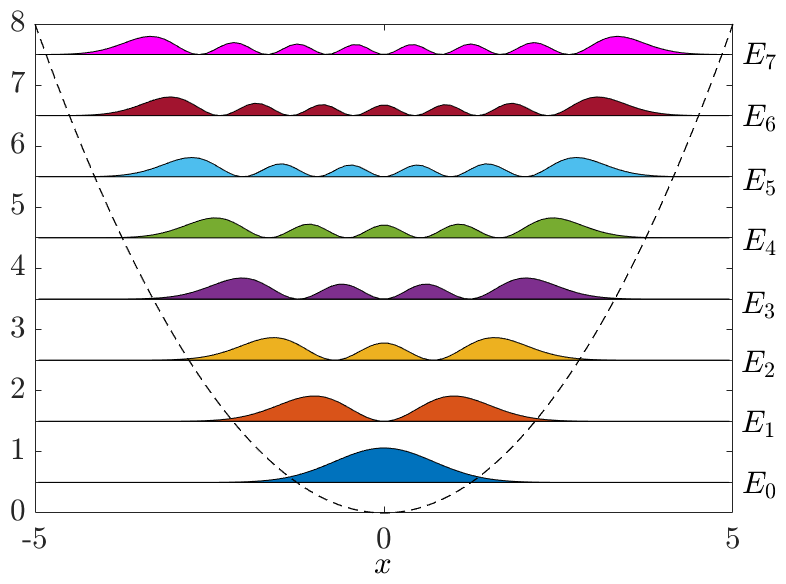}
    \end{minipage}
    \caption{(a) Numerically computed eigenfunctions $ \psi_\ell $ and associated energy levels $ E_\ell $ of the quantum harmonic oscillator. The results are virtually indistinguishable from the analytical results. (b) Corresponding probability densities $ p_\ell $.}
    \label{fig:QHO}
\end{figure}

\end{example}

\begin{example}
As a second example, let us analyze the Schrödinger equation for the hydrogen atom, where $ W(x) = -\frac{e^2}{4 \pi \varepsilon_0 \norm{x}} $, with $ x \in \R^3 $. Here, $ e $ is the electron charge and $ \varepsilon_0 $ the vacuum permittivity. Note that the parameter $ \mathbf{m} $ in front of the Laplacian is the reduced mass of the system. As before, we define the physical constants to be one and use the Gaussian kernel, now with bandwidth $ \varsigma = 2 $. We then generate 5000 uniformly distributed test points in the ball with radius $ 20 $ and compute the Gram matrices $ G_{2} $ and $ G_{0} $. Solving the resulting eigenvalue problem, we obtain the eigenfunctions shown in Figure~\ref{fig:Hydrogen}. As expected, there are several repeated eigenvalues (up to small perturbations due to the randomly sampled test points and numerical errors) for the higher energy states. \exampleSymbol

\begin{figure}
    \centering
    \begin{minipage}{0.1\textwidth}
        \includegraphics[width=0.75\textwidth]{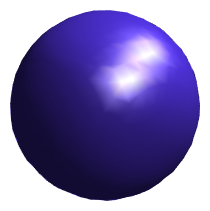}
    \end{minipage}
    \begin{minipage}{0.38\textwidth}
        \centering
        \subfiguretitle{(a)}
        \includegraphics[width=\textwidth]{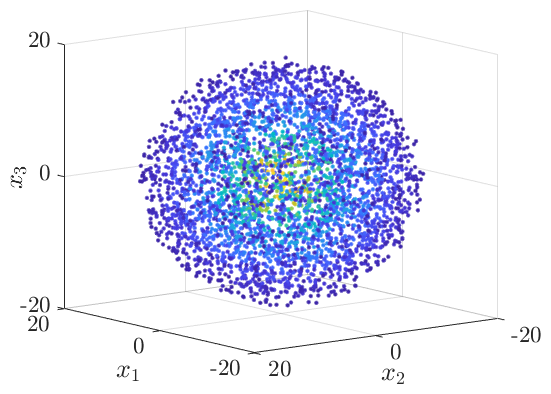}
    \end{minipage}
    \begin{minipage}{0.38\textwidth}
        \centering
        \subfiguretitle{(b)}
        \includegraphics[width=\textwidth]{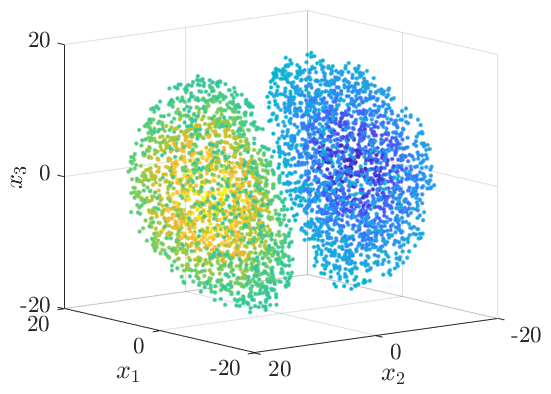}
    \end{minipage}
    \begin{minipage}{0.1\textwidth}
        \includegraphics[width=0.9\textwidth]{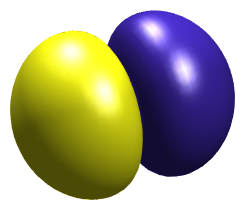}
    \end{minipage} \\[1ex]
    \begin{minipage}{0.1\textwidth}
        \includegraphics[width=0.9\textwidth]{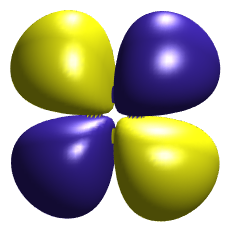}
    \end{minipage}
    \begin{minipage}{0.38\textwidth}
        \centering
        \subfiguretitle{(c)}
        \includegraphics[width=\textwidth]{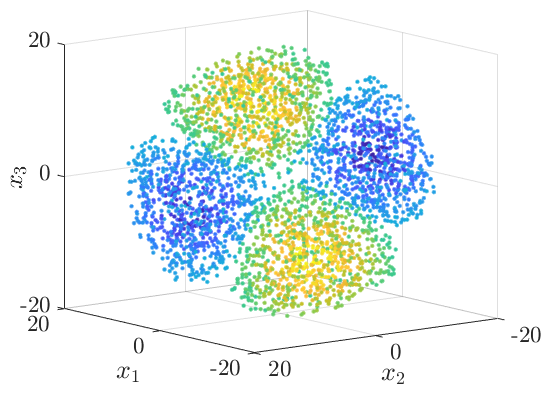}
    \end{minipage}
    \begin{minipage}{0.38\textwidth}
        \centering
        \subfiguretitle{(d)}
        \includegraphics[width=\textwidth]{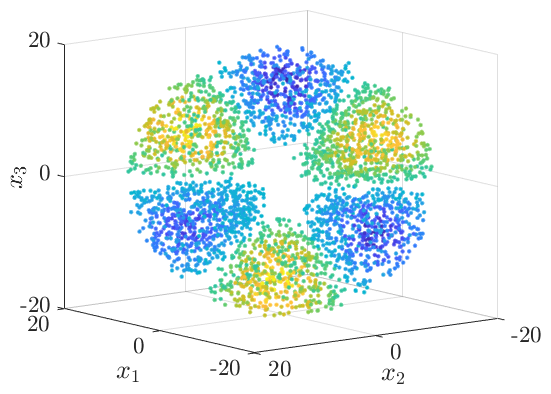}
    \end{minipage}
    \begin{minipage}{0.1\textwidth}
        \includegraphics[width=0.9\textwidth]{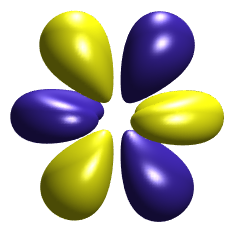}
    \end{minipage}
    \caption{Numerically computed eigenfunctions of the Schrödinger equation associated with the hydrogen atom. Only points where the absolute value of the eigenfunction is larger than a given threshold are plotted. The shapes clearly resemble the well-known hydrogen atom orbitals shown next to the scatter plots. The eigenfunctions (or rotations thereof) correspond to the following quantum numbers $ (\underline{n}, \underline{\ell}, \underline{m}) $: (a) $(1, 0, 0)$, (b) $ (2, 1, 1) $, (c) $ (3, 2, 1) $, (d) $ (4, 3, 1) $.}
    \label{fig:Hydrogen}
\end{figure}

\end{example}

\subsubsection{SDE formulation of the Schrödinger equation}

In order to derive gEDMD, we went from the stochastic differential equation to the Kolmogorov backward equation, which is the generator of the Koopman operator, or the adjoint Fokker--Planck equation, which is the generator of the Perron--Frobenius operator. Exploiting the resemblance between these two equations and the Schrödinger equation, we illustrated how data-driven methods can, in the same way, be used to compute wavefunctions. We now want to go in the opposite direction and find a stochastic differential equation whose eigenfunctions correspond to the wavefunctions. Formal similarities between quantum mechanics and the theory of stochastic processes have been investigated since the beginning of quantum mechanics by Schrödinger and others (see, for example, \cite{Okamoto_1990} and references therein). The necessary transformations were already introduced in Section~\ref{sec:Dynamical operators}, we now want to exploit these relationships. Let us consider the two aforementioned examples again.

\begin{example}
Using Corollary~\ref{corr:KBE_SE}, the quantum harmonic oscillator can be transformed into an Ornstein--Uhlenbeck process
\begin{equation*}
    \mathrm{d}X_t = -\alpha \ts X_t \ts \mathrm{d}t + \sqrt{2 \beta^{-1}} \ts \mathrm{d}B_t,
\end{equation*}
with friction coefficient $ \alpha = \hbar \omega $ and temperature $ \beta^{-1} = \frac{\hbar^2}{2 \mathbf{m}} $. Since the eigenvalues of the Ornstein--Uhlenbeck process are $ \lambda_\ell = -\alpha \ell = -\hbar \omega \ell $, the resulting eigenvalues of the quantum harmonic oscillator are $ E_\ell = E_0 - \lambda_\ell = \hbar \ts \omega \left(\ell + \frac{1}{2}\right) $. Correspondingly, the (unnormalized) eigenfunctions of the Ornstein--Uhlenbeck process are $ \varphi_\ell(x) = \widetilde{H}_\ell(\sqrt{\alpha \beta} \ts x) $, where $ \widetilde{H}_\ell $ is the $\ell$th probabilists’ Hermite polynomial. Thus,
\begin{equation*}
    \psi_\ell(x)
        = \psi_0(x) \widetilde{H}_\ell \left(\sqrt{\frac{2 \mathbf{m} \omega}{\hbar}} \ts x\right)
        = \exp\left(-\frac{\mathbf{m} \omega}{2 \hbar} x^2\right) H_\ell\left(\sqrt{\frac{\mathbf{m} \omega}{\hbar}} \ts x\right),
\end{equation*}
which is consistent with the results obtained above. In the last step, we transformed the probabilists' Hermite polynomials into physicists' Hermite polynomials. \exampleSymbol
\end{example}

\begin{example}
Similarly, for the hydrogen atom, whose ground state is given by
\begin{equation*}
    \psi_0(x) = \frac{1}{\sqrt{\pi a_0^3}} \exp\left(-\frac{1}{a_0}\norm{x}\right),
\end{equation*}
where $ a_0 = \frac{4 \pi \varepsilon_0 \hbar^2}{\mathbf{m} e^2} $, we obtain $ V(x) = \frac{\hbar^2}{\mathbf{m} \ts a_0} \norm{x} $ and thus
\begin{equation*}
    \nabla V(x) = \frac{\hbar^2}{\mathbf{m} \ts a_0} \frac{x}{\norm{x}}.
\end{equation*}
There are now two options to compute the eigenfunctions numerically: We can either directly apply kernel gEDMD to the Koopman generator or generate time-series data by integrating the stochastic differential equation and then apply kernel EDMD or simply Ulam's method. We proceed with the former, but the latter leads to comparable results (although typically more data points are required to achieve the same accuracy due to the stochasticity). We again generate uniformly distributed test points $ x_m $ in the ball with radius $ 20 $, this time $m = 10\ts000$, and use the Gaussian kernel with bandwidth $ \varsigma = 2 $. This results in the same eigenfunctions as the ones shown in Figure~\ref{fig:Hydrogen}. Due to the larger number of test points, even higher energy states can be well-approximated. Two additional eigenfunctions are shown in Figure~\ref{fig:Hydrogen SDE}. \exampleSymbol

\begin{figure}
    \centering
    \begin{minipage}{0.1\textwidth}
        \includegraphics[width=0.8\textwidth]{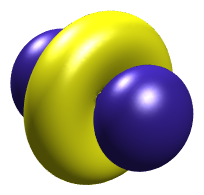}
    \end{minipage}
    \begin{minipage}{0.38\textwidth}
        \centering
        \subfiguretitle{(a)}
        \includegraphics[width=\textwidth]{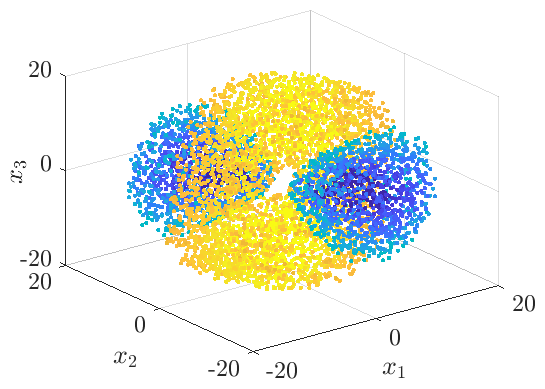}
    \end{minipage}
    \begin{minipage}{0.38\textwidth}
        \centering
        \subfiguretitle{(b)}
        \includegraphics[width=\textwidth]{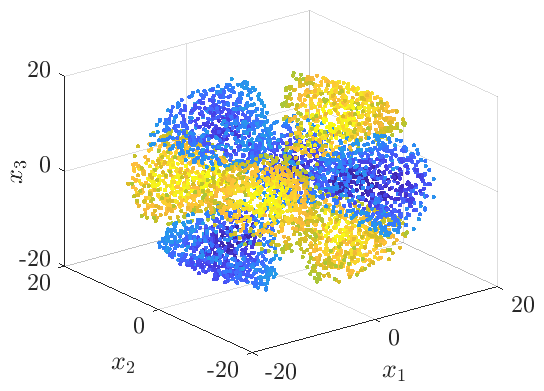}
    \end{minipage}
    \begin{minipage}{0.1\textwidth}
        \includegraphics[width=0.8\textwidth]{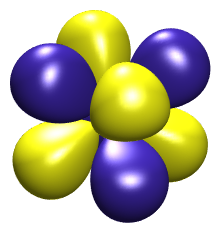}
    \end{minipage}
    \caption{Eigenfunctions of the Schrödinger equation associated with the hydrogen atom computed by applying kernel gEDMD to the corresponding Koopman generator. The quantum numbers $ (\underline{n}, \underline{\ell}, \underline{m}) $ are: (a) $(3, 2, 0)$, (b) $(4, 3, 2)$.}
    \label{fig:Hydrogen SDE}
\end{figure}

\end{example}

The examples illustrate that instead of solving partial differential equations, we can also compute eigenfunctions by approximating the Koopman operator from time-series data. The question under which conditions a non-degenerate strictly positive ground state exists needs to be addressed separately. One important theorem can be found in~\cite{ReedSimonIV}:

\begin{theorem}
Let $ L^2_{\text{loc}}(\mathbb{X}) $ be the space of locally square-integrable functions and $ W \in L^2_{\text{loc}}(\mathbb{X}) $ positive. Suppose $ \lim_{\abs{x} \to \infty} W(X) = \infty $, then $ -\Delta + W $ has a non-degenerate strictly positive ground state.
\end{theorem}

There are other results concerning the existence of such states, see \cite{ReedSimonIV} for details. Also diffusion Monte Carlo methods, which simultaneously compute the ground state energy and wavefunction, rely on similar assumptions~\cite{KFS96}. However, in many cases of interest, the ground state of fermionic systems will have nodes so that these methods are not applicable~\cite{KFS96}. The work presented here aims mainly at linking different operators describing the evolution of dynamical systems, more detailed relationships---in particular with the aforementioned diffusion Monte Carlo methods---and practical implications will be studied in future work.

\subsection{Manifold learning}

So far, we assumed that the data is generated by a dynamical system. There is, however, a second scenario without any notion of time, where the Kolmogorov backward equation and Fokker--Planck equation are used for dimensionality reduction and manifold learning \cite{NLCK06}, see also \cite{Coifman06, CKLMN08, NLCK08} and references therein.

Let the data points $ \{ \ts x_m \ts \}_{m=1}^M $ be sampled from an arbitrary probability density $ \rho $, then we can define the associated potential by
\begin{equation*}
    U(x) = -\log \rho(x).
\end{equation*}
It was shown in \cite{NLCK06} that, depending on some normalization parameter $ \alpha $, anisotropic diffusion maps approximate operators of the form
\begin{equation*}
    \mathcal{L}_\alpha f = -2 (1 - \alpha) \nabla U \cdot \nabla f + \Delta f.
\end{equation*}
That is, for $ \alpha = \frac{1}{2} $, we obtain the standard Kolmogorov backward equation with $ \beta = 1 $. Thus, the algorithms described above could also potentially be used for manifold learning purposes. We will illustrate this with a simple example.

\begin{example}
We consider the well-known Swiss roll, see, for instance, \cite{NLCK08}. The goal is to parametrize the two-dimensional manifold. We use kernel density estimation, cf.~\cite{Parzen62}, and a Gaussian kernel with bandwidth $ \varsigma = 0.22 $ to learn $ U(x) $, i.e.,
\begin{equation*}
    U(x) = \frac{1}{M (\sqrt{2 \pi} \ts \varsigma)^d} \sum_{m=1}^M k(x, x_m)
\end{equation*}
and approximate the backward Kolmogorov operator by applying kernel gEDMD. Here, $ M = 2000 $ and $ d = 2 $. The results are shown in Figure~\ref{fig:SR}. The first eigenfunction parametrizes the angular direction, followed by higher-order modes, and only the sixth eigenfunction corresponds to the $ x_3 $ direction. Considering these eigenfunctions as new coordinates, we obtain an unfolding of the roll. Note that also diffusion maps do not yield perfect rectangles in the embedded space due to the non-uniform density of points on the manifold \cite{NLCK08}. \exampleSymbol

\begin{figure}
    \centering
    \begin{minipage}{0.36\textwidth}
        \centering
        \subfiguretitle{(a)}
        \includegraphics[width=\textwidth]{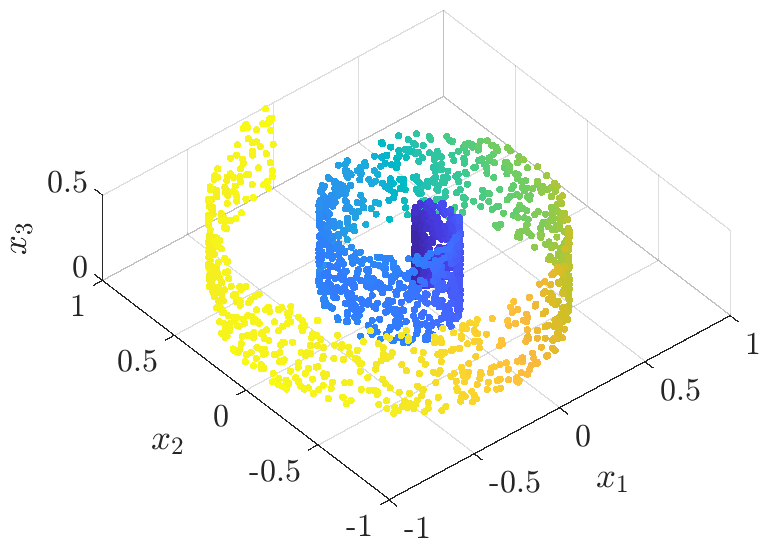}
    \end{minipage}
    \begin{minipage}{0.36\textwidth}
        \centering
        \subfiguretitle{(b)}
        \includegraphics[width=\textwidth]{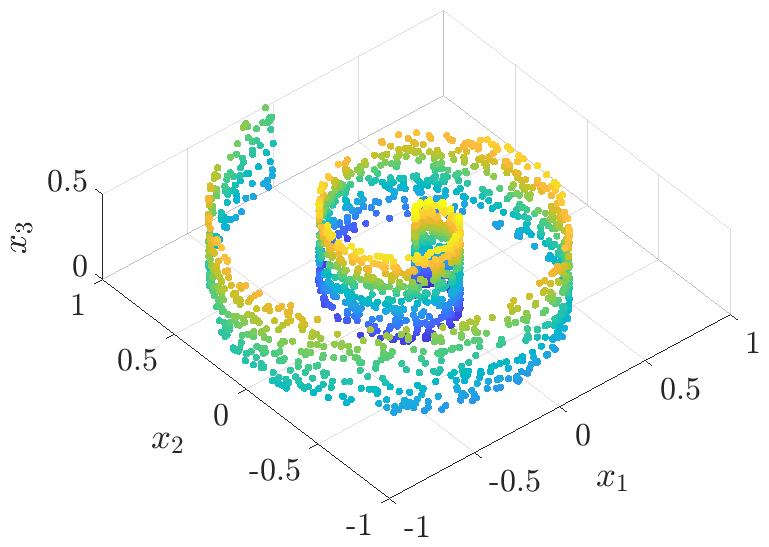}
    \end{minipage}
    \begin{minipage}{0.26\textwidth}
        \centering
        \subfiguretitle{(c)}
        \includegraphics[width=\textwidth]{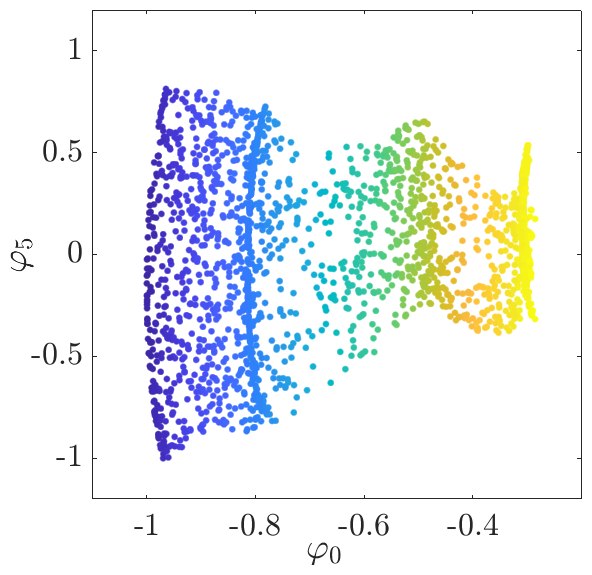}
    \end{minipage}
    \caption{Swiss roll colored with respect to the eigenfunctions (a) $ \varphi_0 $ and (b) $ \varphi_5 $, which parametrize the angular and vertical direction, respectively. (c)~Resulting two-dimensional embedding.}
    \label{fig:SR}
\end{figure}

\end{example}

These results demonstrate that the eigenfunctions of certain differential operators capture geometrical properties of the data. However, the assumption that a strictly positive density in the ambient space exists will in general not be satisfied if the data is supported only on a lower-dimensional manifold. This problem was circumvented by using kernel density estimation and a kernel with global support. Carrying over the definition of the differential operators involved and of their kernel-based analogues to the manifold case is beyond the scope of this paper and will be studied in future work. The same applies to the investigation of detailed relationships with diffusion maps or other manifold learning techniques. Concepts like \emph{neighborhood} and \emph{sparsity} will then need to be carried over to gEDMD to make this method amenable to large data sets. Furthermore, heuristics to find the optimal bandwidth $ \varsigma $ are required since the results often strongly depend on the kernel hyperparameters.

\section{Conclusion}
\label{sec:Conclusion}

Using the theory of derivative reproducing kernel Hilbert spaces, we derived a kernel-based formulation of gEDMD for approximating the Koopman generator, which allows for the computation of eigenfunctions of potentially high-dimensional stochastic dynamical systems. If the system is reversible, the generator can be approximated from equilibrated time-series data, without having to estimate the drift and diffusion terms at the training data points. Furthermore, we showed that data-driven methods developed for the analysis of stochastic dynamical systems (kernel EDMD) can be carried over to their generators (kernel gEDMD) and, in turn, to the Schrödinger operator. Conversely, under certain assumptions on the ground state, the Schrödinger equation can be turned into a Kolmogorov backward equation corresponding to a drift-diffusion process. These results are summarized in Figure~\ref{fig:Relationships}. Similar transformations also exist for the Fokker--Planck operator, see~\cite{Pav14}. All derived approaches were illustrated with numerical results ranging from molecular dynamics to quantum mechanics.

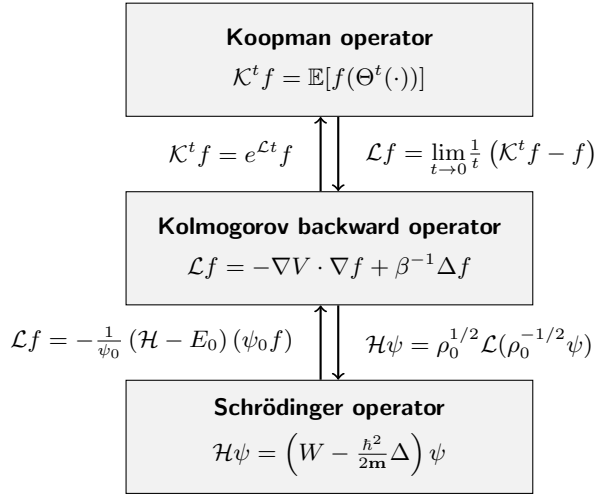
\begin{figure}
    \centering
    \begin{tikzpicture}
    [auto,font=\footnotesize,
     every node/.style={node distance=2.5cm},
     method/.style={rectangle, draw, fill=black!05, inner sep=5pt, text width=5cm, text badly centered, minimum height=1.5cm, font=\bfseries\footnotesize\sffamily}]
    
    \node [method] (KO) {Koopman operator \\[1ex] $ \mathcal{K}^t f = \mathbb{E}[f(\Theta^t(\cdot))] $ };
    
    \node [method, below of=KO] (KBO) { Kolmogorov backward operator \\[1ex] $ \mathcal{L} f =  -\nabla V \cdot \nabla f + \beta^{-1} \Delta f $ };
    
    \node [method, below of=KBO] (SO) { Schrödinger operator \\[1ex] $ \mathcal{H} \psi = \left(W -\frac{\hbar^2}{2 \mathbf{m}} \Delta\right) \psi $ };
    
    \path[->,thick]
    (KO) edge [right, transform canvas={xshift=0.7ex}] node {$~~\mathcal{L} f = \smash{\lim\limits_{t \rightarrow 0}} \frac{1}{t} \left(\mathcal{K}^t f - f \right)$} (KBO);
    
    \path[<-,thick]
    (KO) edge [left, transform canvas={xshift=-0.7ex}] node {$ \mathcal{K}^t f = e^{\mathcal{L}t} f~~$} (KBO);
    
    \path[->,thick]
    (KBO) edge [right, transform canvas={xshift=0.7ex}] node {$~~\mathcal{H}\psi = \rho_0^{1/2} \mathcal{L}(\rho_0^{-1/2} \psi) $} (SO);
    
    \path[<-,thick]
    (KBO) edge [left, transform canvas={xshift=-0.7ex}] node {$ \mathcal{L}f = -\frac{1}{\psi_0}\left(\mathcal{H} - E_0\right)(\psi_0 f)~~$} (SO);
    
    \end{tikzpicture}
    \caption{Relationships between the Koopman, Kolmogorov, and Schrödinger operators for a drift-diffusion process of the form $ \mathrm{d}X_t = -\nabla V(X_t) \ts \mathrm{d}t + \sqrt{2 \beta^{-1}} \ts \mathrm{d}B_t $. Here, $ \rho_0 $ denotes the invariant density, i.e., $ \mathcal{L}^* \rho_0 = 0 $. In our setting, the transformation of the Schrödinger operator requires a strictly positive real-valued ground state $\psi_0$.}
    \label{fig:Relationships}
\end{figure}

Although we focused mainly on the Kolmogorov backward equation, the Fokker--Planck equation, and the Schrödinger equation, these methods can be applied to approximate other differential operators as well. An interesting open question is whether such algorithms can also be used for manifold learning. Some preliminary results were presented in Section~\ref{sec:Applications}, but a rigorous mathematical justification would require significant additional research. Analyzing connections with diffusion maps \cite{Coifman06} or generalizations thereof in detail could be a potential direction for future work.

Another interesting avenue for future research could be to improve the efficiency and stability of the presented algorithms. Exploiting properties of given kernels, it might be possible to speed up computations significantly. The definition of a cutoff radius for the kernel or considering only a certain number of neighbors of data points, for instance, would---for suitable problems---result in sparse matrices. Moreover, the results sensitively depend on the hyperparameters such as the bandwidth of the Gaussian kernel. If the bandwidth is too small, this leads to overfitting and noisy eigenfunctions. If it is, on the other hand, too large, then the kernel is not able to accurately capture the properties of the dynamical system anymore. As a result, the Gram matrix $ G_0 $ has (numerically) essentially a low rank structure and we obtain many zero eigenvalues. The question is then how to efficiently compute the smallest nonzero eigenvalues and corresponding eigenvectors.

Potential solutions for the hyperparameter-tuning problem are techniques based on cross-validation \cite{MP15} or so-called kernel flows \cite{Owhadi19}. By defining an optimization problem for the parameters of the kernel, e.g., based on a variational principle \cite{WuNo17}, gradient descent methods can help find suitable parameter values.

\section*{Acknowledgements}

S.~K.\ was funded by Deutsche Forschungsgemeinschaft (DFG) through grant CRC 1114 (Scaling Cascades in Complex Systems, project ID: 235221301) and through Germany's Excellence Strategy (MATH\texttt{+}: The Berlin Mathematics Research Center, EXC-2046/1, project ID: 390685689). B.~H.\ thanks the European Commission for funding through the Marie Curie fellowships scheme. We would like to thank Luigi Delle Site for many helpful discussions regarding quantum chemistry and Amel Durakovic for useful discussions on the correspondence between the Schrödinger equation and complex Langevin dynamics.

{\small
\bibliographystyle{unsrturl}
\bibliography{kgEDMD}
}

\appendix

\section{Proofs}
\label{sec:Proofs}

\begin{proof}[Proof of Lemma~\ref{lem:KBE_SE}]
The unitary transformation is $\mathcal{H} = -e^{-\frac{1}{2}F} \mathcal{L} \left(e^{\frac{1}{2}F} \cdot \right)$. We obtain
\begin{align*}
\mathcal{H}f &= e^{-\frac{1}{2}F} \mathcal{S} \left(e^{\frac{1}{2}F} f \right) + e^{-\frac{1}{2}F} J \cdot \left(e^{\frac{1}{2}F} \nabla f + \frac{1}{2} e^{\frac{1}{2}F} f \nabla F\right) \\
&= e^{-\frac{1}{2}F} \mathcal{S} \left(e^{\frac{1}{2}F} f \right) + J \cdot \nabla f + \frac{1}{2} J \cdot \nabla F f,
\end{align*}
which establishes the first-order term in $\mathcal{H}$ and the third term in the definition of $W$. For the symmetric part, we find that
\begin{align*}
e^{-\frac{1}{2}F} \mathcal{S} \left(e^{\frac{1}{2}F} f \right) &= -\frac{1}{2}e^{\frac{1}{2}F} \nabla \cdot (e^{-F} a \nabla (e^{\frac{1}{2}F} f)) \\
&= -\frac{1}{2}e^{\frac{1}{2}F} \nabla \cdot (e^{-\frac{1}{2}F} a [\nabla f + \frac{1}{2}\nabla F f]) \\
&= -\frac{1}{2} \nabla \cdot ( a \nabla f) + \frac{1}{4} \nabla F^\top a \nabla f - \frac{1}{4}e^{\frac{1}{2}F} \nabla \cdot (e^{-\frac{1}{2}F} f a \nabla F),
\end{align*}
which establishes the second-order term in the definition of $\mathcal{H}$. Expanding the third term above, we get
\begin{align*}
-\frac{1}{4}e^{\frac{1}{2}F} \nabla \cdot (e^{-\frac{1}{2}F} f a \nabla F) &= -\frac{1}{4} \nabla \cdot (a \nabla F) f - \frac{1}{4} \nabla f^\top a \nabla F + \frac{1}{8} \nabla F^\top a \nabla F f,
\end{align*}
which cancels out the second term of the previous equation and establishes the remaining terms for $W$.

For the converse direction, we first translate the eigenvalue equation for $\psi_0$ into an equation for $\eta$:
\begin{align*}
0 &= (\mathcal{H} - E_0)\psi_0 = -\frac{1}{2}\nabla\cdot \left( a\nabla e^{-\eta} \right) + J \cdot \nabla e^{-\eta} + (W - E_0)e^{-\eta} \\
&= -\frac{1}{2}\nabla \cdot \left( -e^{-\eta} a \nabla \eta \right) - J \cdot \nabla \eta e^{-\eta} + (W - E_0)e^{-\eta} \\
&= e^{-\eta} \left[\frac{1}{2}\nabla \cdot \left( a \nabla \eta \right) - \frac{1}{2}\nabla \eta^\top a \nabla \eta - J \cdot \nabla \eta + W - E_0 \right],
\end{align*}
implying that the term in brackets is also vanishing. Now, we define the negative generator by the transformation $-\mathcal{L} = e^\eta\left[\mathcal{H} - E_0 \right] (e^{-\eta} \cdot)$. Expanding the action of $-\mathcal{L}$, we find
\begin{align*}
-\mathcal{L}f &= e^\eta \left[-\frac{1}{2}\nabla\cdot \left( a\nabla(e^{-\eta}f) \right) + J \cdot \nabla (e^{-\eta}f) + (W - E_0)(e^{-\eta}f) \right] \\
&= e^\eta \left[-\frac{1}{2}\nabla\cdot \left( e^{-\eta} a[ \nabla f - \nabla \eta f] \right) + e^{-\eta} J \cdot \nabla f + (W - E_0 - J \cdot \nabla \eta)(e^{-\eta}f) \right] \\
&= -\frac{1}{2}  \nabla\cdot \left( a \nabla f \right) + \frac{1}{2}\nabla \eta^\top a \nabla f + \frac{1}{2} \nabla \cdot (a\nabla \eta) f + \frac{1}{2} \nabla \eta^\top a \nabla f - \frac{1}{2}\nabla \eta^\top a \nabla \eta f + \\
& \quad J \cdot \nabla f + (W - E_0 - J \cdot \nabla \eta)f \\
&= -\frac{1}{2}  \nabla\cdot \left( a \nabla f \right) + \nabla \eta^\top a \nabla f + J \cdot \nabla f \\
&= -\frac{1}{2}e^{2\eta} \nabla \cdot \left(e^{-2\eta} a \nabla f\right) +  J \cdot \nabla f. \qedhere
\end{align*}
\end{proof}

\begin{proof}[Proof of Lemma~\ref{lem:RKHS Galerkin}]
We only show the proof for $\mathcal{T}_\mathbb{H}$. Similar to the argument in \cite{MUANDET2017}, $\mathcal{T}_{\mathbb{H}}$ is a bounded linear operator on $\mathbb{H}$ because of
\begin{align*}
& \bigg\|\int \phi(x) \otimes \left[ - \frac{1}{2}\sum_{i=1}^d \sum_{j=1}^d a_{ij}(x) D^{e_i + e_j} \phi(x) \right] \ts \mathrm{d} \mu(x) \\
&\quad + \int \phi(x) \otimes \left[\sum_{i = 1}^d \left(J_i(x) - \frac{1}{2} e^{F(x)} \nabla \cdot (e^{-F(x)}a_{:, i}(x))\right) D^{e_i}\phi(x) \right] \ts \mathrm{d} \mu(x) \\
    & \quad + \int  W(x) \phi(x) \otimes \phi(x) \ts \mathrm{d} \mu(x) \bigg\|_{HS} \\
&\leq \frac{1}{2} \sum_{i=1}^d \sum_{j=1}^d \int |a_{ij}(x)| \|D^{e_i + e_j} \phi(x)\|_{\mathbb{H}} \|\phi(x)\|_\mathbb{H}  \ts \mathrm{d} \mu(x) \\
& \quad + \sum_{i=1}^d \int \left(|J_i(x)| + \frac{1}{2} e^{F(x)} |\nabla \cdot (e^{-F(x)}a_{:, i}(x))| \right) \|D^{e_i}\phi(x)\|_\mathbb{H} \|\phi(x)\|_\mathbb{H} \ts \mathrm{d} \mu(x) \\
&\quad + \int  |W(x)| \|\phi(x)\|_\mathbb{H} \|\phi(x)\|_\mathbb{H}  \ts \mathrm{d} \mu(x) \\
&< \infty.
\end{align*}
Using the derivative reproducing property, we obtain
\begin{align*}
    \innerprod{\mathcal{T}f}{g}_{\mu}
        &= \int (\mathcal{T}f)(x) \ts g(x) \ts \mathrm{d} \mu(x) \hfill \\
        &= \int \left[- \frac{1}{2}\sum_{i=1}^d \sum_{j=1}^d a_{ij}(x) \frac{\partial^2 f}{\partial x_i \partial x_j}(x)\right] g(x) \ts \mathrm{d} \mu(x) \\
    & \quad + \int \left[\sum_{i=1}^d \left(J_i(x) - \frac{1}{2} e^{F(x)} \nabla \cdot (e^{-F(x)}a_{:, i}(x))\right) \frac{\partial f}{\partial x_i} \right] g(x) \ts \mathrm{d} \mu(x) \\
    &\quad +  \int W(x) f(x)g(x) \ts \mathrm{d} \mu(x) \hfill \\
    &= \int \left[- \frac{1}{2}\sum_{i=1}^d \sum_{j=1}^d a_{ij}(x) \innerprod{D^{e_i + e_j}\phi(x)}{f}_\mathbb{H}\right] \innerprod{\phi(x)}{g}_\mathbb{H} \ts \mathrm{d} \mu(x) \\
    & \quad + \int \left[\sum_{i=1}^d \left(J_i(x) - \frac{1}{2} e^{F(x)} \nabla \cdot (e^{-F(x)}a_{:, i}(x))\right) \innerprod{D^{e_i}\phi(x)}{f}_\mathbb{H} \right] \innerprod{\phi(x)}{g}_\mathbb{H} \ts \mathrm{d} \mu(x) \\
    &\quad +  \int W(x) \innerprod{\phi(x)}{f}_\mathbb{H} \innerprod{\phi(x)}{g}_\mathbb{H} \ts \mathrm{d} \mu(x) \hfill \\
    &= - \int \left[ \frac{1}{2} \sum_{i=1}^d \sum_{j=1}^d a_{ij}(x) \innerprod{D^{e_i+e_j}\phi(x)\otimes \phi(x)}{f\otimes g}_{\mathbb{H}\otimes \mathbb{H}} \right] \ts \mathrm{d} \mu(x) \\
    & \quad + \int \left[\sum_{i=1}^d \left(J_i(x) - \frac{1}{2} e^{F(x)} \nabla \cdot (e^{-F(x)}a_{:, i}(x))\right) \innerprod{D^{e_i}\phi(x)\otimes \phi(x)}{f\otimes g}_{\mathbb{H}\otimes \mathbb{H}} \right] \ts \mathrm{d} \mu(x) \\
    &\quad +  \int W(x) \innerprod{\phi(x)\otimes \phi(x)}{f\otimes g}_{\mathbb{H}\otimes \mathbb{H}} \ts \mathrm{d} \mu(x) \hfill \\
        &= \innerprod{\mathcal{T}_{\mathbb{H}} f}{g}_\mathbb{H}.
\end{align*}
The same argument can be used to prove the statement about the symmetric case.
\end{proof}

\begin{proof}[Proof of Lemma~\ref{lem:EVP}]
Let $ f = \Phi \ts u $ and $ g = \Phi \ts v $. Then
\begin{align*}
    \widehat{\mathcal{S}}_\mathbb{H}(f, g)
        &= \innerprod{ \left[\frac{1}{M} \sum_{m=1}^M \phi(x_m) \otimes \phi(x_m) \right] \sum_{r=1}^M u_r \ts \phi(x_r)}{\sum_{s=1}^M v_s \ts \phi(x_s)}_\mathbb{H} \\
        &= \frac{1}{M} \sum_{m=1}^M \sum_{r=1}^M \sum_{s=1}^M u_r \ts v_s  \innerprod{\phi(x_m)}{\phi(x_r)}_\mathbb{H} \innerprod{\phi(x_m)}{\phi(x_s)}_\mathbb{H} \\
        &= \frac{1}{M} \sum_{m=1}^M \sum_{r=1}^M \sum_{s=1}^M u_r \ts v_s  \ts k(x_m, x_r) \ts k(x_m, x_s) \\
        &= \frac{1}{M} \innerprod{G_{0} \ts u}{G_{0} \ts v}.
\end{align*}
Similarly,
\begin{align*}
    \widehat{\mathcal{Q}}_\mathbb{H}(f, g)
        &= \innerprod{ \left[\frac{1}{M} \sum_{m=1}^M \phi(x_m) \otimes \mathrm{d}\phi(x_m) \right] \sum_{r=1}^M u_r \ts \phi(x_r)}{\sum_{s=1}^M v_s \ts \phi(x_s)}_\mathbb{H} \\
        &= \frac{1}{M} \sum_{m=1}^M \sum_{r=1}^M \sum_{s=1}^M u_r \ts v_s  \innerprod{\mathrm{d}\phi(x_m)}{\phi(x_r)}_\mathbb{H} \innerprod{\phi(x_m)}{\phi(x_s)}_\mathbb{H} \\
        &= \frac{1}{M} \innerprod{G_{2} \ts u}{G_{0} \ts v}.
\end{align*}
If the kernel functions at the training points are linearly independent, then $ G_{0} $ is invertible, and it suffices to compute eigenvectors $ u $ of the generalized matrix eigenvalue problem $ G_{2} \ts u = \lambda \ts G_{0} \ts u $. In the symmetric case, the expression for the quadratic form $ \widehat{\mathcal{Q}}_\mathbb{H}(f, g)$ changes to
\begin{align*}
    \widehat{\mathcal{Q}}_\mathbb{H}(f, g)
        &= \frac{1}{2 \ts M} \sum_{l=1}^d \sum_{m=1}^M \sum_{r=1}^M \sum_{s=1}^M u_r \ts v_s \left(\sigma_{l}(x_m)^\top \nabla k(x_m, x_r)\right) \left(\sigma_{l}(x_m)^\top \nabla k(x_m, x_s)\right) \\
        &= \frac{1}{2 \ts M} \sum_{l=1}^d \innerprod{G_{1}^{(l)} \ts u}{G_{1}^{(l)} \ts v}. \qedhere
\end{align*}
\end{proof}

\end{document}